\documentclass{amsart}
\usepackage{tikz}
\usepackage{subcaption}
\usetikzlibrary{arrows.meta}
\usepackage{graphicx} 
\usepackage{amsfonts, amsmath,amssymb,amsthm}
\usepackage{thmtools}
\usepackage{xcolor}
\usepackage{soul}
\usepackage[shortlabels]{enumitem}
\usepackage[hypertexnames=false]{hyperref}
\usepackage{cleveref}
\AddToHook{cmd/appendix/before}{\crefalias{section}{appendix}}

\crefname{enumi}{}{}

\usepackage{tikz}
\usepackage{graphicx}
\usepackage[normalem]{ulem}
\usepackage{fontawesome5}
\usepackage{mathrsfs}
\usepackage{longtable}
\usepackage{array}
\usepackage{booktabs}
\usepackage{multirow}

\crefformat{enumi}{(\mathrm{#2#1#3})}
\Crefformat{enumi}{(\mathrm{#2#1#3})}
\crefname{enumi}{}{ }

\usepackage[margin=1 in]{geometry}

\newcommand{\rk}{\mathrm{rk}}

\newcommand{\GHV}{the GHV lower bound}

\newcommand{\KK}{\mathbb{K}}

\newcommand{\RR}{\mathbb{R}}

\newcommand{\ZZ}{\mathbb{Z}}

\newcommand{\cB}{\mathcal{B}}
\newcommand{\C}{{\mathcal{C}}}

\newcommand{\cI}{\mathcal{I}}

\newcommand{\cL}{\mathcal{L}}

\newcommand{\si}{\mathrm{si}}
\newcommand{\cl}{\mathrm{cl}}

\newcommand{\cu}{\mathrm{c}}
\newcommand{\g}{\mathrm{g}}

\newcommand{\NP}{\mathrm{NP}}
\newcommand{\SP}{\mathrm{SP}}

\theoremstyle{plain}
\newtheorem{Theorem}{Theorem}[section]
\newtheorem{Lemma}[Theorem]{Lemma}
\newtheorem{Proposition}[Theorem]{Proposition}
\newtheorem{Observation}[Theorem]{Observation}
\newtheorem{Corollary}[Theorem]{Corollary}

\theoremstyle{definition}
\newtheorem{Definition}[Theorem]{Definition}
\newtheorem{Example}[Theorem]{Example}
\newtheorem{Question}[Theorem]{Question}

\newtheorem{Remark}[Theorem]{Remark}

\providecommand{\customgenericname}{}
\newtheorem{innercustomgeneric}{\customgenericname}
\providecommand{\customgenericname}{}
\newcommand{\newcustomtheorem}[2]{%
  \newenvironment{#1}[1]
  {%
   \renewcommand\customgenericname{#2}%
   \renewcommand\theinnercustomgeneric{##1}%
   \innercustomgeneric
  }
  {\endinnercustomgeneric}
}
\def\urltilda{\kern -.15em\lower .7ex\hbox{\~{}}\kern .04em} 
\newcustomtheorem{customthm}{Theorem}
\newcustomtheorem{acknowledgement}{Acknowledgement}

\newcommand{\cir}{\mathcal{C}}

\begin{document}

\title{Asymptotic Resurgence of Facet ideals of Graphic Matroids}

\author{Michael DiPasquale}
\address{Department of Mathematical Sciences \\
New Mexico State University\\
P.O. Box 30001 \\
Department 3MB \\
Las Cruces, NM 88003}
\email{midipasq@nmsu.edu}
\urladdr{\href{https://midipasq.github.io}{https://midipasq.github.io}}

\author{Louiza Fouli}
\address{Department of Mathematical Sciences \\
New Mexico State University\\
P.O. Box 30001 \\
Department 3MB \\
Las Cruces, NM 88003}
\email{lfouli@nmsu.edu}
\urladdr{\href{https://sites.google.com/view/louiza-fouli/home}{\tt https://sites.google.com/view/louiza-fouli}}

\author{Arvind Kumar}
\address{Department of Mathematical Sciences \\
New Mexico State University\\
P.O. Box 30001 \\
Department 3MB \\
Las Cruces, NM 88003}
\email{arvkumar@nmsu.edu}
\urladdr{\href{https://sites.google.com/view/arvkumar/home}{https://sites.google.com/view/arvkumar/home}}

\begin{abstract}
Our main results are an upper bound on the asymptotic resurgence of the facet ideal of a graphic matroid in terms of the number of vertices and a lower bound in terms of the circumference.  These bounds coincide for Hamiltonian graphs, which form the majority of graphs on $n$ vertices as $n$ tends to infinity.  For simple $2$-connected graphs on up to nine vertices, we compute in Sage that the asymptotic resurgence of the facet ideal of a non-Hamiltonian graphic matroid is given either by our upper or lower bound.
\end{abstract}



\maketitle

\section{Introduction}

If $I$ is a radical ideal in a polynomial ring $S=\KK[x_1,\ldots,x_n]$ over a field $\KK$ of characteristic $0$, the $s$-th symbolic power of $I$ is defined by
$
I^{(s)}=\bigcap\limits_{P\in \mbox{Min}(I)} (I^sS_P\cap S),
$
where $\mbox{Min}(I)$ is the set of minimal primes of $I$.  According to the Zariski-Nagata theorem, $I^{(s)}$ consists of those polynomials that vanish to order $s$ along the variety defined by $I$.  A great deal of literature is devoted to studying symbolic powers and their relationship to ordinary powers; see the survey~\cite{DDGHN2018}.

To aid in the comparison of ordinary and symbolic powers of an ideal $I$, Bocci-Harbourne~\cite{BH10} and Guardo-Harbourne-Van Tuyl~\cite{GHV13} introduced the numerical invariants known as the \textit{resurgence} $\rho(I)$ and \textit{asymptotic resurgence} $\widehat{\rho}(I)$ defined by
\[
\rho(I)=\sup\left\lbrace\frac{s}{r}~:~ I^{(s)}\not\subset I^r\right\rbrace \quad \mbox{and}\quad 
\widehat{\rho}(I)=\sup\left\lbrace\frac{s}{r}~:~ I^{st}\not\subset I^{rt} \mbox{ for all } t\gg 0\right\rbrace,
\]
where $s,r,t$ are positive integers. It is clear from the definitions that $\widehat{\rho}(I)\le \rho(I)$, while the celebrated uniform containment theorems of Ein--Lazarsfeld--Smith~\cite{ELS01}, Hochster--Huneke~\cite{HH02}, and Ma--Schwede~\cite{MS17} imply the general upper bound that $\rho(I)\le h$, where $h$ denotes the largest codimension among the associated primes of $I$.  Computing these invariants is notoriously difficult, and explicit formulas are currently known only for a relatively small number of families of ideals, many of which arise from constructions which are combinatorial~\cite{MG18,DFMS19,DiPasquale-Drabkin-2021,JKM22,DFK26} or geometric~\cite{BH10,BH2010,DHNSST2015,GHMN17,BDHHSS2019, KM2026}.

The present paper builds upon our work in \cite{DFK26}, where we initiated a systematic study of the asymptotic resurgence of \emph{facet ideals of matroids}. In \cite{DFK26}, we proved that the asymptotic resurgence of the facet ideal of a matroid
can be expressed in terms of the \emph{Waldschmidt constant} (see~\Cref{sec:Background}) of facet ideals of contractions of the matroid.  We used this contraction formula to derive bounds on the asymptotic resurgence of facet ideals of matroids and to compute the asymptotic resurgence exactly for some families of matroids.

This paper builds upon the framework of \cite{DFK26} in two directions.  First, we develop structural reduction theorems in \Cref{sec:reduction-techniques} that simplify the computation of the asymptotic resurgence.  Namely, we prove that simplification preserves the asymptotic resurgence of the facet ideal of a matroid (\Cref{thm:simplification}) and that the asymptotic resurgence of the facet ideal of a matroid is the maximum of the asymptotic resurgences of the facet ideals of its connected components (\Cref{thm:direct-sum}).  
We also establish additional bounds in \Cref{sec:walds-upper-bounds}, some of which are incomparable with those obtained in~\cite{DFK26} (see Remarks~\ref{rem:incomp1} and ~\ref{rem:incomp2}).

Second, we begin a study of the asymptotic resurgence of facet ideals of graphic matroids in \Cref{sec:graphic-matroids}. In \cite{DFK26} we obtained general bounds for the asymptotic resurgence of the facet ideal of a matroid in terms of the size of the ground set and the rank of the matroid. When considering the facet ideals of graphic matroids, the underlying graph provides additional combinatorial tools that allow us to sharpen the bounds in \cite{DFK26}, often yielding the exact asymptotic resurgence.  Concretely, if $G$ is a graph with the vertex set $V(G)$ and the edge set $E(G)$, the \textit{graphic matroid} $M=M(G)$ is the matroid on the ground set $E(G)$ with bases corresponding to spanning forests of $G$.  Let $S=\KK[x_e:e\in E(G)]$.  The facet ideal $I(M)$ is the ideal
\[
I(M)=\left\langle \prod_{e \in T} x_e~:~ T \mbox{ a spanning forest of } G\right\rangle.
\]
Two of our main results, \Cref{thm:asym-gr-mat-up} and \Cref{cor:asym-lb-grap}, yield that, if $G$ is a simple $2$-connected graph,
\[
\frac{2(\cu(G)-1)}{\cu(G)}\le \widehat{\rho}(I(M(G)))\le \frac{2(|V(G)|-1)}{|V(G)|},
\]
where $V(G)$ is the vertex set of $G$ and $\cu(G)$ is the \textit{circumference} of $G$.  The upper and lower bounds coincide when $\cu(G)=|V(G)|$; that is, $G$ is Hamiltonian.  These bounds bear a striking resemblance to the asymptotic resurgence of edge and cover ideals of graphs: it follows from \cite{DFMS19} and \cite{Villarreal-2023} that
\[
\frac{2(\omega(G) -1)}{\omega(G)} \leq \widehat{\rho}(I(G))=\frac{2(\chi_f(G)-1)}{\chi_f(G)}=\widehat{\rho}(J(G)) \leq \frac{2(\chi(G)-1)}{\chi(G)},\]
where $I(G)$ is the edge ideal of $G$, $J(G)$ is the cover ideal of $G$, $\omega(G)$ is the clique number of $G$, $\chi(G)$ is the chromatics number of $G$ and $\chi_f(G)$ is the fractional chromatic number of $G$.  Despite the clear resemblance, it is not clear that the asymptotic resurgence of the facet ideal of a graphic matroid admits such a clean characterization, as we indicate through a number of examples and questions in \Cref{sec:examples,sec:Conclusion}.

We close the introduction with some motivation for studying the asymptotic resurgence of facet ideals of matroids. A theorem of Villareal \cite{Villarreal-2023} shows that $\widehat{\rho}(I(M))=\widehat{\rho}(I_{\Delta(M^*)})$ (see \cite[Proposition~2.6]{DFK26}).
Consequently, every result on the asymptotic resurgence of facet ideals of matroids immediately translates into a corresponding result for the asymptotic resurgence of Stanley-Reisner ideals of their dual matroids, and hence to an important collection of projective schemes called \textit{matroid configurations}~\cite{GHMN17}.  In particular, our results in \Cref{sec:graphic-matroids} can be directly applied to \textit{cographic matroid configurations}.  We pursue this connection in forthcoming work.

Another motivation for studying the asymptotic resurgence of the facet ideal $I(M)$ of a matroid $M$ is that it often achieves equality in the lower bound
$
\frac{\alpha(I(M))}{\widehat{\alpha}(I(M))}\le \widehat{\rho}(I(M)).
$
We refer to this lower bound as \textit{\GHV{}} since it was derived by Guardo-Harbourne-Van Tuyl in~\cite{GHV13}.  Quite a few equigenerated squarefree monomial ideals attain equality in \GHV{} (e.g. edge ideals of graphs~\cite{DFMS19}).  In \cite{DFK26} we prove that $\widehat{\rho}(I(M))$ achieves equality in \GHV{} for sparse paving matroids whose rank is at most half the size of the ground set; conjecturally, most matroids on a ground set of size $n$ fall into this class as $n$ increases (see \cite[Conjectures~1.6 and 1.10]{MNDW11} and \cite[Conjecture~5.15]{DFK26}).  Our work in this paper continues the trend - if $G$ is Hamiltonian, our results imply that $\widehat{\rho}(I(M))$ achieves equality in \GHV{} and it is well-known that as $n$ grows, the overwhelming majority of graphs on $n$ vertices are Hamiltonian \cite{BB83}.

\section{Background}\label{sec:Background}

Let $S=\KK[x_1,\ldots,x_n]$
be a polynomial ring over a field $\KK$. Throughout this article, all ideals are homogeneous. For an ideal $I\subset S$, we denote by $\mathrm{Min}(I)$ the set of minimal primes of $I$. If $I$ is a squarefree monomial ideal, then its $s$-th symbolic power satisfies
\[
I^{(s)}
=
\bigcap_{P\in \mathrm{Min}(I)} P^s,
\] 
see, for instance, \cite{CEHH17}. For a positive integer $n$, we set $[n]:=\{1,\ldots,n\}$. If ${\bf a}=(a_1,\ldots,a_n)\in\ZZ_{\geq 0}^n$, then we write $x^{\bf a}:=x_1^{a_1}\cdots x_n^{a_n}$. For $f_1,\ldots,f_k\in S$, we write $\langle f_1,\ldots,f_k\rangle$ for the ideal generated by $f_1,\ldots,f_k$. For a subset $U\subseteq [n]$, we define
\[
x^U:=\prod_{i\in U}x_i
\qquad \text{and} \qquad
P_U:=\langle x_i ~:~ i\in U\rangle.
\]

We now introduce matroids and their facet ideals.

\subsection{Matroids and their facet ideals} A \emph{matroid} $M=(E,\cI)$ consists of a finite set $E$ together with a collection $\cI$ of subsets of $E$ satisfying the following axioms:
\begin{enumerate}
\item $\emptyset\in \cI$.
\item If $\sigma\in\cI$ and $\sigma'\subseteq \sigma$, then $\sigma'\in\cI$.
\item If $\sigma,\sigma'\in\cI$ and $|\sigma|>|\sigma'|$, then there exists an element $e\in \sigma\setminus \sigma'$ such that $\sigma'\cup\{e\}\in\cI$.
\end{enumerate}
We call $E$ the \emph{ground set} of $M$, and we call the elements of $\cI$ the \emph{independent sets} of $M$.  The maximal independent sets of $M$ are called the \emph{bases} of $M$. We denote the collection of bases of $M$ by $\cB(M)$.
A subset $D\subseteq E$ is \emph{dependent} if $D\notin \cI$. The minimal dependent sets are called the \emph{circuits} of $M$, and we denote the collection of circuits of $M$ by $\cir(M)$. We call an element $e\in E$ a \emph{loop} of $M$ if $e$ is not contained in any basis of $M$. We call two distinct elements $e,f\in E$ \emph{parallel} if $\{e,f\}$ is a circuit of $M$. We call $M$ \emph{simple} if it has no loops and no parallel elements. Standard references for matroids, which we use throughout the paper, are~\cite{Oxley2011} and~\cite{Wel}.

The \emph{rank function} of a matroid $M$ is the map 
$\rk_M:2^E\to \ZZ_{\geq 0}$ defined by $$\rk_M(A)=\max\{|A\cap B|~:~ B\in \cB(M)\}$$ for every subset $A\subseteq E$. We call $\rk_M(A)$ the \emph{rank} of $A$ in $M$. The rank of the ground set $E$ is called the \emph{rank of the matroid}, and we denote it by $\rk(M)$. Equivalently, $\rk(M)$ is the common cardinality of every basis of $M$. Throughout the article, we assume $\rk(M)<|E|$.

Given a subset $A\subseteq E$, we define the \emph{closure} of $A$ in $M$, denoted by $\cl_M(A)$, to be the maximal subset of $E$ containing $A$ with the same rank as $A$. We call a subset $F\subseteq E$ a \emph{flat} of $M$ if $\cl_M(F)=F$. We denote the collection of flats of $M$ by $\cL(M)$, and the collection of flats of rank $i$ by $\cL_i(M)$. We call a flat $F$ a hyperplane if $\rk_M(F)=\rk(M)-1$. The collection $\cL(M)$ forms a geometric lattice under inclusion. Given two flats $F_1,F_2\in \cL(M)$, their meet and join are $F_1\wedge F_2 = F_1\cap F_2$ and 
$F_1\vee F_2 = \cl_M (F_1\cup F_2)$, respectively. 

If $M=(E,\cB)$ is a matroid, then the \emph{dual matroid} $M^*$ is the matroid on the same ground set whose bases are
\[
\cB(M^*)
=
\{E\setminus B ~:~ B\in \cB(M)\}.
\]
Let $U\subseteq E$. The \emph{deletion} of $U$ from $M$, denoted by $M\backslash U$, is the matroid on the ground set $E\setminus U$ whose bases are
\[
\cB(M\backslash U)
=
\{B\in \cB(M)~:~ B\cap U=\emptyset\}.
\]
We define the \emph{contraction} of $U$ from $M$, denoted by $M/U$, by $M/U := (M^*\backslash U)^*$.

Let $M$ be a matroid on a ground set $E$ with $|E|=n$, and let $S=\KK[x_e~:~ e\in E]$. Without loss of generality, we may assume $E=[n]$, and $S=\KK[x_1,\ldots, x_n]$. The \emph{facet ideal}, or \emph{basis ideal} of $M$ is
\[
I(M):=
\langle x^B ~:~ B\in \cB(M)\rangle=\bigcap\limits_{C\in \cir(M^*)}P_C,
\]
where the prime decomposition follows, for instance, from Alexander duality applied to the Stanley-Reisner ideal of the independence complex of $M^*$ (c.f. \cite[Section~2]{DFK26}).

In~\cite{CEHH17}, Cooper-Embree-Hà-Hoefel define the \textit{symbolic polyhedron} of a monomial ideal. We restrict to the squarefree case.  For a positive integer $n$, let $\RR_{\geq 0}^n$ denote the set of vectors in $\RR^n$ with nonnegative coordinates, and let ${\bf e}_1,\ldots,{\bf e}_n$ denote the standard basis vectors of $\mathbb{R}^n$. For $U\subseteq [n]$, we define $\chi_U:=\sum\limits_{i\in U} {\bf e}_i\in {\RR^n}$.  If $E$ is a finite set with $|E|=n$, we may also write $\RR^E$ instead of $\RR^n$, with the understanding that the standard basis vectors are indexed by elements of $E$.

\begin{Definition}[\cite{CEHH17}]\label{def:SymbolicPolyhedron}
Let $I$ be a squarefree monomial ideal of $S$. The symbolic polyhedron of $I$ is
\[
\SP(I):=\left\{{\bf a}=(a_1,\ldots,a_n)\in\RR^n_{\ge 0}~:~ \chi_U \cdot {\bf a} =\sum_{i\in U} a_i \ge 1\mbox{ for all } P_U\in\mbox{Min}(I)\right\}.
\]
\end{Definition}

Since $I(M)=\bigcap\limits_{C\in\C(M^*)} P_C$, the symbolic polyhedron of $I(M)$ is determined by the circuits of $M^*$. 
\begin{Corollary}\label{cor:SPIM}
Let $M$ be a matroid on the ground set $E$.  The defining inequalities of $\SP(I(M))$ are given by $ \chi_C\cdot {\bf a}=\sum\limits_{i\in C}a_i\ge 1$ for every circuit $C\in \C(M^*)$.
\end{Corollary}

The following theorem from \cite{CEHH17} expresses the Waldschmidt constant of $I$ in terms of optimization over its symbolic polyhedron.

\begin{Theorem}[{\cite[Corollary~6.3]{CEHH17}}]\label{thm:SPWaldschmidt}
Let $I$ be a squarefree monomial ideal of $S$.  Then
\[
\widehat{\alpha}(I)=\min\left\{\chi_{[n]} \cdot {\bf a}=\sum_{i=1}^n a_i~:~ {\bf a}=(a_1,\ldots,a_n)\in \SP(I)\right\}.
\]
\end{Theorem}

\subsection{Prior results on Waldschmidt constant and asymptotic resurgence}

We collect here a few results which we will use from ~\cite{DFK26}.

\begin{Proposition}\label{prop:minors}
Let $M$ be a matroid on the ground set $E$ and let $U\subseteq E$.  Then
\[
\widehat{\rho}(I(M))\ge \max\{\widehat{\rho}(I(M/U)),\widehat{\rho}(I(M\backslash U))\}.
\]
\end{Proposition}
\begin{proof}
This follows from \cite[Corollary~3.4]{DFK26} and \cite[Remark~3.7]{DFK26}.
\end{proof}

\begin{Theorem}\cite[Proposition~3.8, Theorem~3.9]{DFK26}\label{thm:asym-res-facet}
Let $M$ be a matroid on the ground set $E$.  Then, 
\[
\widehat{\rho}(I(M))=
\max\limits_{F \in  \cL(M) \setminus \{E\}}\left\{ \frac{\rk(M/F)}{\widehat{\alpha}(I(M/F))}  \right \}=\max\left\lbrace \frac{\rk(M)}{\widehat{\alpha}(I(M))},\widehat{\rho}_{c,1}(I(M))\right\rbrace,
\]
where $\widehat{\rho}_{c,1}(I(M))=\max\limits_{F\in \cL_1(M)}\{\widehat{\rho}(I(M/F))\}$.
\end{Theorem}

Recall that the uniform matroid $\mathrm{U}_{k,n}$ is the rank $k$ matroid on a ground set $E$ with $|E|=n$ whose bases are precisely the $k$-subsets of $E$. Equivalently, when $k<n$, the circuits of $\mathrm{U}_{k,n}$ are exactly the subsets of $E$ of cardinality $k+1$.

\begin{Proposition}\cite[Corollary 5.7]{DFK26}\label{prop:asym-upper-bound}
Let $M$ be a non-uniform matroid on the ground set $E$. Then, 
\[\widehat{\rho}(I(M)) \le 
   \max \left\{  \frac{\rk(M)(|E|-\rk(M))}{|E|-1}, 
    \frac{(\rk(M)-1)(|E|-\rk(M)+1)}{|E|-1} \right\}.\]
\end{Proposition}

\section{Reduction principles for the asymptotic resurgence of facet ideals of matroids}\label{sec:reduction-techniques} 

In this section, we develop reduction principles that allow us to reduce the study of asymptotic resurgence of facet ideals of matroids to simple (\Cref{thm:simplification}) and connected matroids (\Cref{thm:direct-sum}). 

\subsection{Reduction to simple matroids:} Let $M$ be a matroid of rank $k$ on the ground set $E$. A \emph{parallel extension} $N$ of $M$ is a single-element extension obtained by adding a new element $f$ in parallel to an existing element $e \in E$. If $e$ is a loop in $M$, then both $e$ and $f$ are loops in $N$. If $\rk_M(\{e\}) = 1$, then $\{e,f\}$ is a circuit of $N$, and more generally, for every $g \in \text{cl}_M(\{e\})$, $\{g,f\}$ is a circuit of $N$. The rank of the parallel extension $N$ is the same as that of $M$, that is, $\rk(N) = \rk(M) = k$. For each $0 \leq r \leq k$, the flats of $N$ of rank $r$ are precisely:
\begin{itemize}
    \item the flats $F \subseteq E$ of $M$ of rank $r$ with $e \notin F$, and
    \item the sets $F \cup \{f\}$, where $F$ is a flat of $M$ of rank $r$ with $e \in F$.
\end{itemize}

\begin{Theorem}\label{thm:para-ext}
Let $M$ be a matroid on the ground set $E$, and let $N$ be a parallel extension of $M$.  Then, $$\widehat{\alpha}(I(N))=\widehat{\alpha}(I(M)).$$
\end{Theorem}
\begin{proof}
Let $N$ be a parallel extension of $M$ obtained by adding a new element $f$ in parallel to an element $e \in E$.  We use \Cref{thm:SPWaldschmidt} to show that $\widehat{\alpha}(I(N))=\widehat{\alpha}(I(M))$. By \Cref{cor:SPIM}, the symbolic polyhedron $\SP(I(M))$ is defined by the inequalities
$ \chi_C \cdot \mathbf{a}  = \sum\limits_{i \in C} a_i \geq 1$ for all  $C \in \mathcal{C}(M^*),$ 
and $\SP(I(N))$ is defined by $ \chi_C \cdot \mathbf{a} = \sum\limits_{i \in C} a_i \geq 1$ for all  $C \in \mathcal{C}(N^*)$. 

The hyperplanes of $N$ are the hyperplanes of $M$ not containing $e$, and the sets $H \cup \{f\}$, where $H$ is a hyperplane of $M$ with $e \in H$, hence $\C(N^*)=\{C \in \mathcal{C}(M^*):  e \notin C\} \cup\{C \cup \{f\}: C \in \mathcal{C}(M^*), e \in C\}$.  Therefore, the defining inequalities of $\SP(I(N))$ are:
\[
\sum_{i \in C} a_i \geq 1 
\quad \text{for all } C \in \mathcal{C}(M^*) \text{ with } e \notin C,
\]
and
\[
a_f + \sum_{i \in C} a_i \geq 1 
\quad \text{for all } C \in \mathcal{C}(M^*) \text{ with } e \in C.
\]

Next, let ${\bf v} \in \SP(I(M))$ be such that
$\widehat{\alpha}(I(M)) = \sum\limits_{i \in E} v_i$. Define ${\bf \widetilde{v}} \in \mathbb{R}^{E \cup \{f\}}_{\geq 0} $ by $\widetilde{v}_g = v_g \text{ for } g \in E,$ and $\widetilde{v}_f = 0$. Then,  $\bf \widetilde{v}$ satisfies all defining inequalities of $\SP(I(N))$, hence ${\bf \widetilde{v} }  \in \SP(I(N))$. Therefore,
\[
\widehat{\alpha}(I(N)) \leq \sum_{i \in E \cup \{f\}} \widetilde{v}_i 
= \sum_{i \in E} v_i 
= \widehat{\alpha}(I(M)).
\]

Conversely, let ${ \bf u} \in \SP(I(N)) $ be such that
$\widehat{\alpha}(I(N)) = \sum\limits_{i \in E \cup \{f\}} u_i$.
Define ${\bf \overline{u}} \in \mathbb{R}^E_{\geq 0} $ by
$\overline{u}_g = u_g \text{ for all } g \in E \setminus \{e\}$ and  $\overline{u}_e = u_e + u_f$. Then, ${\bf \overline{u}}$ satisfies all defining inequalities of $\SP(I(M))$, and hence ${\bf \overline{u}} \in \SP(I(M))$. Consequently,
\[
\widehat{\alpha}(I(M)) \leq \sum_{i \in E} \overline{u}_i 
= \sum_{i \in E \cup \{f\}} u_i 
= \widehat{\alpha}(I(N)).
\]
Hence, $\widehat{\alpha}(I(N)) = \widehat{\alpha}(I(M))$.
\end{proof}

Let $M$ be a matroid on the ground set $E$. The \emph{simplification} of $M$, denoted $\mathrm{si}(M)$, is obtained as follows:
\begin{itemize}
\item for any flat $F\neq \emptyset$ of rank $0$, delete all the elements of $F$, and
\item for each flat $F$ of rank $1$, retain exactly one element of $F$ and delete all others.
\end{itemize}
 
Notice that if $M$ is a loopless matroid, then $M$ is obtained from $\si(M)$ through a series of parallel extensions. Furthermore, $\mathrm{si}(M)$ is the unique simple matroid whose lattice of flats is isomorphic to that of $M$ (see~\cite[Chapter 1]{Oxley2011}).

\begin{Theorem}\label{thm:simplification}
Let $M$ be a matroid of rank $k$ on the ground set $E$. Then, 
\begin{enumerate}
    \item $\widehat{\alpha}(I(M/\mathrm{cl}_M(F)))=\widehat{\alpha}(I(\mathrm{si}(M)/F))$ for every $0 \leq i \leq k-1$ and every $F \in \cL_i(\mathrm{si}(M))$. 
    \item $\widehat{\rho}(I(M)) =\widehat{\rho}(I(\mathrm{si}(M))). $
\end{enumerate}
\end{Theorem}
\begin{proof}
Let $0 \le i \le k-1$ and let $F \in \cL_i(\si(M))$. Let $G=\mathrm{cl}_M(F)$. We claim that $\si(M/G)=\si(\si(M)/F)$. Since $\si(M)$ is the unique simple matroid whose lattice of flats is isomorphic to that of $M$, the flats of $\si(M)$ are naturally identified with the flats of $M$. Under this identification, the flat $F$ of $\si(M)$ corresponds to the flat $G=\mathrm{cl}_M(F)$ of $M$. 
Now, contraction by a flat is determined at the level of lattices of flats: if two matroids have isomorphic lattices of flats and corresponding flats are contracted, then the resulting contraction lattices are again isomorphic. Therefore, $\cL(M/G)\cong \cL(\si(M)/F)$. 
Taking simplifications on both sides, and using again the fact that simplification is the unique simple matroid with the same lattice of flats, we obtain $\si(M/G)=\si(\si(M)/F)$. 

Now, it follows from \Cref{thm:para-ext} that \[\widehat{\alpha}(I(M/G))=\widehat{\alpha}(I(\si(M/G)))=\widehat{\alpha}(I(\si(\si(M)/F)))=\widehat{\alpha}(I(\si(M)/F)).\]

By \Cref{thm:asym-res-facet}, \[
\widehat{\rho}(I(M))
= \max_{\substack{0\leq i \leq k-1\\ G \in \cL_i(M)}} \left\{
\frac{\rk(M/G)}{\widehat{\alpha}(I(M/G))}\right\}
= \max_{\substack{0 \le i \le k-1,\\ F \in \cL_i(\si(M))}} 
\left\{\frac{\rk(\si(M)/F)}{\widehat{\alpha}(I(\si(M)/F))}\right\}
= \widehat{\rho}(I(\si(M))). \qedhere
\]  
\end{proof}

\subsection{Reduction to connected matroids:}

Let $M_1$ and $M_2$ be matroids on the disjoint ground sets $E_1$ and $E_2$, respectively. The \emph{direct sum} of the matroids  $M_1$ and $M_2$ is another matroid, denoted  $M_1 \oplus M_2$, on the ground set $E_1 \cup E_2$, see \cite[Chapter 5.3]{Wel}. The bases of $M_1 \oplus M_2$ are $\cB(M_1 \oplus M_2)=\{ B_1 \cup B_2~:~ B_i \in \cB(M_i)\}$. Furthermore, the dual of the direct sum is isomorphic to the direct sum of the duals, specifically $(M_1 \oplus M_2)^*=M_1^* \oplus M_2^*$. A matroid is \emph{connected} if it cannot be decomposed as a direct sum.

\begin{Theorem}\label{thm:direct-sum}
Let $M_1$ and $M_2$ be matroids on the disjoint ground sets $E_1$ and $E_2$, respectively, and let $M=M_1 \oplus M_2$. Then, we have the following: 
\begin{enumerate}
    \item $I(M)=I(M_1)\cdot I(M_2)$.
    \item $\widehat{\alpha}(I(M))=\widehat{\alpha}(I(M_1))+\widehat{\alpha}(I(M_2))$.
    \item $\widehat{\rho}(I(M))=\max\left\{\widehat{\rho}(I(M_1)),\widehat{\rho}(I(M_2))\right\}$.
\end{enumerate}
\end{Theorem}
\begin{proof}
Part (a) follows from the definition of the facet ideal of a matroid and the description of the basis of matroid $M$ in terms of the basis of $M_1$ and $M_2$. Specifically,  \begin{align*}
     I(M)&=\langle x^B~:~ B \in \cB(M)\rangle\\ & =\langle x^{B_1 \cup B_2}~:~ B_i \in \cB(M_i)\rangle \\ & =\langle x^{B_1} \cdot x^{B_2}~:~ B_i \in \cB(M_i)\rangle \text{ as } E_1 \cap E_2 =\emptyset \text{ and } B_i \subseteq E_i \\ & =I(M_1)\cdot I(M_2).
 \end{align*} 
 
 \par Part (b): It follows from \cite[Proposition 3.5]{JKM22} that $I(M)^{(s)} =I(M_1)^{(s)} \cdot I(M_2)^{(s)}$ for all $ s \geq 1$. Therefore, $\alpha(I(M)^{(s)})= \alpha(I(M_1)^{(s)})+\alpha(I(M_2)^{(s)})$ for all $s \geq 1$, and hence, $\widehat{\alpha}(I(M))=\widehat{\alpha}(I(M_1))+\widehat{\alpha}(I(M_2))$.  

 \par Part (c) follows from part (a) and \cite[Proposition 3.5]{JKM22}.  
\end{proof}

The final result of this section shows that the facet ideal of a sum of matroids only attains equality in \GHV{} if both summands attain equality with the same value. 
The following elementary observation will be useful to us in the next proposition. 

\begin{Observation} \label{obs: fractions}
    If $a,b,c,d$ are positive real numbers with $\frac{a}{b}\geq \frac{c}{d}$, then
$\frac{a}{b}\geq \frac{a+c}{b+d}\geq \frac{c}{d}$. Moreover,
$\frac{a}{b}=\frac{c}{d}$ if and only if
$\frac{a}{b}=\frac{a+c}{b+d}$ if and only if $\frac{c}{d}=\frac{a+c}{b+d}$.
\end{Observation}

\begin{Proposition}\label{prop:sumequality}
Let $M_1$ and $M_2$ be matroids on the disjoint ground sets $E_1$ and $E_2$, respectively, and let $M=M_1 \oplus M_2$. Then,  $\widehat{\rho}(I(M))=\frac{\alpha(I(M))}{\widehat{\alpha}(I(M))}$ if and only if $\widehat{\rho}(I(M_1))=\frac{\alpha(I(M_1))}{\widehat{\alpha}(I(M_1))}=\frac{\alpha(I(M_2))}{\widehat{\alpha}(I(M_2))}=\widehat{\rho}(I(M_2))$.
\end{Proposition}

\begin{proof}
Assume that $\widehat{\rho}(I(M))=\frac{\alpha(I(M))}{\widehat{\alpha}(I(M))}$.  By Theorem \ref{thm:direct-sum} part (c), we have $$\widehat{\rho}(I(M))=\frac{\alpha(I(M))}{\widehat{\alpha}(I(M))}= \max\left\{\widehat{\rho}(I(M_1)),\widehat{\rho}(I(M_2))\right\}.$$ By \Cref{thm:direct-sum} parts (a-b), we get $\alpha(I(M))= \alpha(I(M_1))+\alpha(I(M_2))$ and $\widehat{\alpha}(I(M))=\widehat{\alpha}(I(M_1))+\widehat{{\alpha}}(I(M_2))$.

We have $\frac{\alpha(I(M_1))+\alpha(I(M_2))}{\widehat{\alpha}(I(M_1))+\widehat{{\alpha}}(I(M_2))}=\frac{\alpha(I(M))}{\widehat{\alpha}(I(M))}=\widehat{\rho}(I(M))\ge\widehat{\rho}(I(M_1))\ge \frac{\alpha(I(M_1)}{\widehat{\alpha}(I(M_1)}$.  Likewise, $\frac{\alpha(I(M_1))+\alpha(I(M_2))}{\widehat{\alpha}(I(M_1))+\widehat{{\alpha}}(I(M_2))}\ge\widehat{\rho}(I(M_2))\ge \frac{\alpha(I(M_2)}{\widehat{\alpha}(I(M_2)}$.  By \Cref{obs: fractions}, we must have $\frac{\alpha(I(M_1))}{\widehat{\alpha}(I(M_1))}={}\frac{\alpha(I(M_1))+\alpha(I(M_2))}{\widehat{\alpha}(I(M_1))+\widehat{{\alpha}}(I(M_2))}=\frac{\alpha(I(M_2))}{\widehat{\alpha}(I(M_2))}$.  This forces
$\widehat{\rho}(I(M_1))=\frac{\alpha(I(M_1))}{\widehat{\alpha}(I(M_1))}=\frac{\alpha(I(M_2))}{\widehat{\alpha}(I(M_2))}=\widehat{\rho}(I(M_2))$.

Conversely, assume that $\widehat{\rho}(I(M_1))=\frac{\alpha(I(M_1))}{\widehat{\alpha}(I(M_1))}=\frac{\alpha(I(M_2))}{\widehat{\alpha}(I(M_2))}=\widehat{\rho}(I(M_2))$. By \Cref{thm:direct-sum} part (c), we have $\widehat{\rho}(I(M))=\widehat{\rho}(I(M_1))=\frac{\alpha(I(M_1))}{\widehat{\alpha}(I(M_1))}=\frac{\alpha(I(M_2))}{\widehat{\alpha}(I(M_2))}=\widehat{\rho}(I(M_2))$ and now rearranging using parts (a-b) of \Cref{thm:direct-sum} and \Cref{obs: fractions}, we get  $\widehat{\rho}(I(M))=\frac{\alpha(I(M))}{\widehat{\alpha}(I(M))}$. 
\end{proof}

\section{Bounds on the Waldschmidt constant and Asymptotic Resurgence}\label{sec:walds-upper-bounds}
In this section, we obtain bounds for the Waldschmidt constant and asymptotic resurgence of facet ideals of matroids. In general, these new bounds are incomparable to those we found in~\cite{DFK26}.

We begin by deriving an upper bound for the Waldschmidt constant of facet ideals in terms of the \emph{girth} of the dual matroid. Recall that the girth of a matroid $M$, denoted by $\g(M)$, is the minimum cardinality of a circuit of $M$. In particular, $\g(\mathrm{U}_{k,n})=k+1$. A matroid $M$ is called \emph{paving} if $\g(M) \geq \rk(M)$.

\begin{Proposition}\label{thm:wald-gen-up} 
If $M$ is a matroid on the ground set $E$, then $\widehat{\alpha}(I(M)) \leq \frac{|E|}{\g(M^*)}$. Moreover, if $M^*$ is a non-uniform paving matroid, then $\widehat{\alpha}(I(M)) \leq \frac{|E|}{\rk(M^*)}$.
\end{Proposition}
\begin{proof}
We claim that $\tfrac{1}{\g(M^*)}\chi_E \in \SP(I(M))$. Let $C \in \mathcal{C}(M^*)$. Then, we have $\chi_{C} \cdot \tfrac{1}{\g(M^*)}\chi_E= \tfrac{|C|}{\g(M^*)} \geq 1$ as $|C| \geq \g(M^*)$. Therefore, $\tfrac{1}{\g(M^*)}\chi_E \in \SP(I(M))$, and hence, by \Cref{thm:SPWaldschmidt},  $\widehat{\alpha}(I(M)) \leq  \tfrac{|E|}{\g(M^*)}$.

Next, if $M^*$ is a non-uniform paving matroid, then $\g(M^*)=\rk(M^*)$. This completes the proof.
\end{proof}

The \emph{circumference} of a matroid $M$, denoted by $\cu(M)$, is the maximum size of a circuit in $M$. A matroid $M$ is said to be \emph{Hamiltonian} if it contains a circuit of size $\rk(M)+1$, or equivalently, if $\cu(M)=\rk(M)+1$.

\begin{Theorem}\label{thm:Hamil-Mat-ub}
Let $M$ be a matroid on the ground set $E$. Then, $\widehat{\alpha}(I(M)) \leq \rk(M)+1-\frac{\cu(M)}{2}$.  In particular, if $M$ is Hamiltonian, then $\widehat{\alpha}(I(M)) \leq \frac{\cu(M)}{2}$. 
\end{Theorem}
\begin{proof}
Let $C$ be a circuit of $M$ with $|C|=\cu(M)$ and fix $e\in C$. Since $C$ is a minimal dependent set, $C\setminus \{e\}$ is independent. Thus, there exists a basis $B$ of $M$
 such that $C\setminus {e}  \subseteq B$. Observe that $C$ is the fundamental circuit of $B$ with respect to $e$, that is,  
$C$ is the unique circuit contained in $B\cup \{e\}$.

Set $N=M\setminus (E\setminus (B \cup \{e\}))$ and $N'=M\setminus (E \setminus C)$.  Then $N$ is a matroid on the ground set $B\cup \{e\}$ with the same rank as $M$ and
$\mathcal{C}(N)=\{C\}$, while $N'$ is a rank $|C|-1$ matroid on the ground set $C$ with $\mathcal{C}(N')=\{C\}$. In particular, $N'=\mathrm{U}_{|C|-1,|C|}$.   By \cite[Proposition 4.2]{DFK26}, $\widehat{\alpha}(I(N')) = \frac{|C|}{2}=\frac{\cu(M)}{2}$.  (The Waldschmidt constant of the facet ideal of a uniform matroid is a well-known result -- see \cite{BH10,BCSGHJNSVV2016} -- but in the literature apart from \cite{DFK26} it is stated for the Stanley-Reisner ideal).  Moreover, $N=N' \oplus \mathrm{U}_{r,r},$ where $r=\rk(M)-\cu(M)+1$.  Thus, by \Cref{thm:direct-sum} and \cite[Proposition 4.2]{DFK26}, $\widehat{\alpha}(I(N)) =\widehat{\alpha}(I(N'))+\widehat{\alpha}(I(\mathrm{U}_{r,r})) = \frac{\cu(M)}{2}+r= \rk(M)+1-\frac{\cu(M)}{2}$.

Finally, since $I(N)\subseteq I(M)$, we have $\widehat{\alpha}(I(M)) \leq \widehat{\alpha}(I(N)) \leq \rk(M)+1-\frac{\cu(M)}{2}$. Moreover, if $M$ is Hamiltonian, then $\cu(M)=\rk(M)+1$, which completes the proof.
\end{proof}

The next lemma is an immediate consequence of \cite[Theorem~1.2]{BocciCooperGuardoHarbourneJanssenNagelSeceleanuVanTuylVu}.

\begin{Lemma}\label{lem:Chudnovsky}
Let $M$ be a matroid on the ground set $E$. Then, $
\widehat{\alpha}(I(M))\ge \frac{\rk(M)+\cu(M^*)-1}{\cu(M^*)}
$.
\end{Lemma}
\begin{proof}
It follows from~\cite[Theorem~1.2]{BocciCooperGuardoHarbourneJanssenNagelSeceleanuVanTuylVu}, that the Waldschmidt constant of a squarefree monomial ideal $I$ satisfies the Chudnovsky-like lower bound
$
\widehat{\alpha}(I)\ge \frac{\alpha(I)+e-1}{e},
$
where $e$ is the maximum height of a minimal prime of $I$.  The minimal primes of $I(M)$ are of the form $P_C=\langle x_i: i\in C\rangle$ where $C$ is a circuit of $M^*$.  So the maximum height of a minimal prime of $I(M)$ is the largest size of a circuit of $M^*$, which is $\cu(M^*)$.
\end{proof}

The following lemma is implicit in \cite{DFK26}.

\begin{Lemma}\label{lem:nonuniform}
Let $M$ be a non-uniform matroid on a ground set $E$. Then
\[
\widehat{\alpha}(I(M))\ge \min\left\lbrace\frac{|E|-1}{|E|-\rk(M)},\frac{|E|+1}{|E|-\rk(M)+1}\right\rbrace.
\]
\end{Lemma}
\begin{proof}
Since $M$ is non-uniform, there is a subset $C\subset E$ of size $\rk(M)$ that is not a basis of $M$.  Thus $I(M)\subset J$, where $J:=\langle x^U: U\subset E, |U|=\rk(M), U\neq C\rangle$.  It follows from \cite[Proposition~5.5]{DFK26} that $\widehat{\alpha}(J)=\min\left\lbrace\frac{|E|-1}{|E|-\rk(M)},\frac{|E|+1}{|E|-\rk(M)+1}\right\rbrace$.  Since $I(M)\subseteq J$, $\widehat{\alpha}(J)\le \widehat{\alpha}(I)$, which completes the proof.
\end{proof}

\begin{Remark}\label{rem:incomp1}
In general, the two bounds of \Cref{lem:Chudnovsky} and \Cref{lem:nonuniform} are incomparable.  For instance, if $M$ is a non-uniform matroid so that $\cu(M^*)=|E|-\rk(M)+1$, then \Cref{lem:nonuniform} gives a better lower bound.  On the other hand, suppose that $M$ is the matroid of the \textit{projective plane geometry} $PG(2,q)$ over a field of characteristic $q$ (see~\cite[Chapter~6.1]{Oxley2011}).  Then $\rk(M)=3$, $\cu(M^*)=q^2$, and $|E|=q^2+q+1$.  \Cref{lem:Chudnovsky} gives $\widehat{\alpha}(I(M))\ge\dfrac{q^2+2}{q^2}$ while \Cref{lem:nonuniform} gives $\widehat{\alpha}(I(M))\ge\dfrac{q^2+q}{q^2+q-2}$, so the lower bound from \Cref{lem:Chudnovsky} is better.  In fact, $\widehat{\alpha}(I(M))=\dfrac{q^2+q+1}{q^2}$ by \cite[Proposition~6.2 and Example~6.5]{DFK26}.
\end{Remark}

\begin{Theorem}\label{thm:AsymptoticResurgenceBoundsMatroid}
Let $M$ be a matroid on a ground set $E$.  Then
\[
2-\frac{2}{\cu(M)}\le \widehat{\rho}(I(M))\le \frac{\rk(M)\cu(M^*)}{\rk(M)+\cu(M^*)-1}.
\]
\end{Theorem}
\begin{proof}
For the lower bound, pick a circuit $C$ of $M$ so that $|C|=\cu(M)$.  Then delete all elements of the ground set outside of the circuit to get the matroid $N=M\backslash(E\backslash C)$.  As we saw in the proof of \Cref{thm:Hamil-Mat-ub}, $N$ is the uniform matroid $\mathrm{U}_{|C|-1,|C|}$.  It follows from \cite[Proposition~4.2]{DFK26} that $\widehat{\rho}(I(N))=\frac{2(|C|-1)}{|C|}=2-\frac{2}{|C|}=2-\frac{2}{\cu(M)}$ (again, this is well-known but usually stated for the Stanley-Reisner ideal of a uniform matroid - see \cite{MG18}).  Now the lower bound follows from \Cref{prop:minors}.

For the upper bound, we use \Cref{thm:asym-res-facet}.  Let $F$ be any flat of $M$.  It follows from \Cref{lem:Chudnovsky} that
\[
\frac{\rk(M/F)}{\widehat{\alpha}(I(M/F))}\le \frac{\rk(M/F)\cu((M/F)^*)}{\rk(M/F)+\cu((M/F)^*)-1}=\frac{\rk(M/F)\cu(M^*\backslash F)}{\rk(M/F)+\cu(M^*\backslash F)-1}
\]
Clearly, $\rk(M/F)\le \rk(M)$.  Circuits of a deletion of a matroid are also circuits of the original matroid.  Hence $\cu(M^*\backslash F)\le \cu(M^*)$.  Put $r=\rk(M)$ and $c=\cu(M^*)$.  Suppose that $1\le r'\le r$ and $1\le c' \le c$.  Then,  it is straightforward to check that
\[
\frac{rc}{r+c-1}\geq\frac{rc'}{r+c'-1}\ge \frac{r'c'}{r'+c'-1}.
\]
By \Cref{thm:asym-res-facet}, $\widehat{\rho}(I(M))\le \frac{rc}{r+c-1}$, as desired.
\end{proof}

\begin{Remark}\label{rem:incomp2}
The upper bound on $\widehat{\rho}(I(M))$ in \Cref{thm:AsymptoticResurgenceBoundsMatroid} is, in general, incomparable to the upper bound in \Cref{prop:asym-upper-bound}.  As in \Cref{rem:incomp1}, if $M$ is a non-uniform matroid so that $\cu(M^*)=|E|-\rk(M)+1$, then the upper bound in \Cref{prop:asym-upper-bound} is better than the upper bound in \Cref{thm:AsymptoticResurgenceBoundsMatroid}.  However, if $M$ is the matroid of the projective plane geometry $PG(2,q)$, so $\rk(M)=3$, $\cu(M^*)=q^2$, and $|E|=q^2+q+1$, then \Cref{thm:AsymptoticResurgenceBoundsMatroid} yields $\widehat{\rho}(I(M))\le \frac{3q^2}{q^2+2}$ while \Cref{prop:asym-upper-bound} yields $\widehat{\rho}(I(M))\le \frac{3(q^2+q-2)}{q^2+q}$.  If $q>2$, then the bound from \Cref{thm:AsymptoticResurgenceBoundsMatroid} is slightly better.  In fact, $\widehat{\rho}(I(M))=\frac{3q^2}{q^2+q+1}$ by \cite[Example~6.5]{DFK26}.
\end{Remark}

\begin{Remark}\label{rem:refinement}
It is curious to compare the upper bound on $\widehat{\rho}(I(M))$ in \Cref{thm:AsymptoticResurgenceBoundsMatroid} to two bounds in the literature: the celebrated uniform containment bound ~\cite{ELS01,HH02,MS17} implies that $\widehat{\rho}(I(M))\le \cu(M^*)$ and a bound for squarefree monomial ideals \cite[Theorem~3.18]{DFMS19} or \cite[Corollary~3.6]{HT-2019} implies that $\widehat{\rho}(I(M))\le \rk(M)$.  The upper bound in \Cref{thm:AsymptoticResurgenceBoundsMatroid} is a simultaneous refinement of both of these.
\end{Remark}

In the remainder of the paper, we shift our attention to graphic matroids, where we significantly improve the upper bounds from \Cref{prop:asym-upper-bound} and \Cref{thm:AsymptoticResurgenceBoundsMatroid} -- see \Cref{thm:asym-gr-mat-up} and \Cref{rem:AsymptoticComparison}.

\section{Waldschmidt constant and asymptotic resurgence of facet ideals of graphic matroids}\label{sec:graphic-matroids}
In this section, we study the Waldschmidt constant and asymptotic resurgence of facet ideals of graphic matroids. We establish bounds for these invariants (\Cref{thm:grap-mat-wald-lb} and  \Cref{thm:asym-gr-mat-up}) in terms of graph-theoretic parameters and show that the bounds are asymptotically sharp. In particular, since almost every graph is Hamiltonian as the number of vertices tends to infinity (see \cite{BB83}),  asymptotically almost every graphic matroid attains these bounds. We begin by recalling the graph-theoretic terminology and the definitions of graphic matroids and their dual matroids that we will use throughout the section.

Let $G$ be a finite graph with the vertex set $V(G)$ and the edge set $E(G)$. A \emph{cycle} in $G$ is a closed walk in which no vertex (and hence no edge) is repeated except for the initial and terminal vertex. A \emph{forest} is an acyclic graph, and a \emph{tree} is a connected forest. If $G$ is connected, a \emph{spanning tree} is a connected, acyclic subgraph with vertex set $V(G)$. Any spanning tree $T$ satisfies $|E(T)| = |V(G)| - 1$. More generally, if $G$ has connected components $G_1,\dots, G_r$, a \emph{spanning forest} $F$ is the disjoint union of spanning trees of $G_i$'s, and hence satisfies $|E(F)| = |V(G)| - r$. A graph $G$ is \emph{simple} if it has no loops or parallel edges, and \emph{$2$-connected} if it remains connected after deletion of any vertex.

Let $G$ be a graph and let $e=\{u,v\}\in E(G)$. To \emph{contract} the edge $e$, we identify the vertices $u$ and $v$ and delete the edge $e$. We denote the resulting graph by $G/e$ and call it the \emph{edge contraction} of $G$ along $e$. Edge contraction does not necessarily preserve simplicity. In particular, if $e$ lies on a triangle of $G$, then contracting $e$ creates parallel edges, even when $G$ is simple. Moreover, edge contraction does not preserve $2$-connectivity in general. We denote by $\si(G/e)$ the simplification of $G/e$, obtained by replacing each set of parallel edges with a single edge and deleting loops.

The \emph{graphic matroid} $M(G)$ is the matroid on the ground set $E(G)$ whose independent sets are the subsets of $E(G)$ that form acyclic subgraphs of $G$. The bases of $M(G)$ are the spanning forests of $G$, and in particular the spanning trees when $G$ is connected. The circuits of $M(G)$ are precisely edge sets of cycles in $G$. If $G$ is connected, then $\rk(M(G)) = |V(G)| - 1$, and in general $\rk(M(G)) = |V(G)| - r$, where $r$ is the number of connected components of $G$. Note that $M(G)$ is connected if and only if $G$ is $2$-connected, and $M(G)$ is simple if and only if $G$ is simple.

The \emph{cographic matroid} $M^*(G)$, also called the \emph{bond matroid} or \emph{cocycle matroid}, is the dual of $M(G)$. Its bases are the complements of the bases of $M(G)$. The circuits of $M^*(G)$ correspond to minimal edge cuts of $G$, called \emph{bonds}. An edge cut is a set of edges whose removal increases the number of connected components, and it is minimal if no proper subset has this property. For a vertex $v \in V(G)$, the set of edges incident to $v$ always forms an edge cut; when $G$ is $2$-connected, this set is minimal and hence defines a circuit of $M^*(G)$, called a \emph{vertex bond}.  If $G$ is not $2$-connected, the set of edges incident to any cut vertex is not a minimal edge cut and hence is not a bond. In light of the reduction results in \Cref{sec:reduction-techniques}, we will only consider simple $2$-connected graphs.

\subsection{Waldschmidt Constant of Facet Ideals of  Graphic Matroids}-
Recall that if $M=M(G)$ is the graphic matroid of a graph $G$, then $\g(M)=\g(G)$, where $\g(G)$ denotes the \emph{girth} of the graph $G$, namely, the length of a shortest cycle in $G$. Dually, if $M=M^*(G)$ is the cographic matroid of $G$, then $\g(M)=\lambda(G)$, where $\lambda(G)$ is the \emph{edge-connectivity} of $G$, that is, the minimum size of a bond of $G$. For a graph $G$, let $\delta(G)$ denote the minimum degree of a vertex of $G$.  A graph is called \emph
{regular} if all its vertices have the same degree. Finally, recall that $\cu(G)$ denotes the circumference of $G$, that is the longest length of a (simple) cycle in $G$.

\begin{Corollary}\label{cor:wald-up-gra}
\begin{enumerate}
    \item If $M=M(G)$ is the graphic matroid of a connected graph $G$, then
    $\widehat{\alpha}(I(M)) \leq \frac{|E(G)|}{\lambda(G)}$.
    \item If $M=M^*(G)$ is the cographic matroid of a graph $G$, then $\widehat{\alpha}(I(M)) \leq \frac{|E(G)|}{\g(G)}$.
\end{enumerate}
\end{Corollary} 
\begin{proof}
If $M=M(G)$, then $M^*$ is the cographic matroid of $G$, so $\g(M^*)=\lambda(G)$. If $M=M^*(G)$, then $M^*=M(G)$, and hence $\g(M^*)=\g(G)$. Therefore, the proof immediately follows from \Cref{thm:wald-gen-up}.
\end{proof}

\begin{Theorem}\label{thm:grap-mat-wald-lb}
Let $G$ be a $2$-connected simple graph, and let $M=M(G)$ be its graphic matroid. Then:
\begin{enumerate}
    \item $\frac{|V(G)|}{2} \leq \widehat{\alpha}(I(M)) \leq |V(G)|-\frac{\cu(G)}{2}$.
    \item If $G$ is Hamiltonian, then $\widehat{\alpha}(I(M))=\frac{|V(G)|}{2}$.
    \item If $G$ is regular and $\lambda(G)=\delta(G)$, then $\widehat{\alpha}(I(M))
    =\frac{|V(G)|}{2}
    =\frac{|E(G)|}{\lambda(G)}$.
\end{enumerate}
\end{Theorem}
\begin{proof}
(a) We use \Cref{thm:SPWaldschmidt} to prove that $\widehat{\alpha}(I(M)) \geq \frac{|V(G)|}{2}$. By \Cref{cor:SPIM}, the symbolic polyhedron $\SP(I(M))$ is defined by the inequalities
$\chi_C \cdot \mathbf{a}  = \sum\limits_{i \in C} a_i \geq 1$ for all  $C \in \mathcal{C}(M^*)$. Since the circuits of $M^*$ are precisely the bonds of $G$, these inequalities may be written as $\chi_B \cdot  \mathbf{a} = \sum\limits_{i \in B} a_i \geq 1$ for all bonds $B \subseteq E(G)$. In particular, this holds for all vertex bonds. As $G$ is $2$-connected, the set of edges incident to any vertex forms a bond.  Moreover, in a simple graph, each edge is incident to exactly two vertices. Therefore, by adding the inequalities corresponding to all vertex bonds, we get
\[ 2\cdot \sum\limits_{e \in E(G)} a_e \geq |V(G)|.\] This implies $\sum\limits_{e \in E(G)} a_e \geq \frac{|V(G)|}{2}$. Hence, by \Cref{thm:SPWaldschmidt},  $\widehat{\alpha}(I(M)) \geq \frac{|V(G)|}{2}$. The upper bound follows from \Cref{thm:Hamil-Mat-ub}, and the fact that $\rk(M)=|V(G)|-1$ and $\cu(M)=\cu(G)$. 

(b) If $G$ is Hamiltonian, $\cu(G)=|V(G)|$, so $\widehat{\alpha}(I(M))=\frac{|V(G)|}{2}$ by (a).

(c) Suppose that $G$ is regular and $\lambda(G)=\delta(G)$. Since $G$ is regular, by the handshaking lemma, \[2|E(G)|=\sum\limits_{v\in V(G)} \deg(v)=\delta(G)|V(G)|=\lambda(G)|V(G)|,\] so $\frac{|E(G)|}{\lambda(G)}= \frac{|V(G)|}{2}$. By \Cref{thm:wald-gen-up}, $\widehat{\alpha}(I(M)) \leq  \frac{|E(G)|}{\lambda(G)}=\frac{|V(G)|}{2}$.
Together with part~(a), this yields $\widehat{\alpha}(I(M)) =\frac{|V(G)|}{2}=\frac{|E(G)|}{\lambda(G)}$. This completes the proof.
\end{proof}

\begin{Remark}
The lower bound in \Cref{thm:grap-mat-wald-lb}~(a) improves on the lower bounds in \Cref{lem:Chudnovsky} and \Cref{lem:nonuniform}.  We compare the subsequent bounds on asymptotic resurgence in \Cref{rem:AsymptoticComparison}.

The upper bound in \Cref{thm:grap-mat-wald-lb}~(a) is sharp, as it is attained by graphic matroids of Hamiltonian graphs. In the next section, we show that the graphic matroids of Theta graphs $\Theta_{2,b,c}$ (\Cref{theta-graphs}) and Jahangir graphs (\Cref{jahangir-graphs}) also realize this bound.

We also point out in \Cref{ex:cubic-planar-lambda-two} that the hypothesis $\lambda(G)=\delta(G)$ in part~(c) of \Cref{thm:grap-mat-wald-lb} is essential. Regularity alone does not force equality in the lower bound.
\end{Remark}

The next result shows that if a graph $G$ achieves the lower bound in \Cref{thm:grap-mat-wald-lb}, then any graph formed by adding edges to $G$ (without increasing the number of vertices) also achieves the lower bound in \Cref{thm:grap-mat-wald-lb}.

\begin{Corollary}
Let $G$ be a simple $2$-connected graph.  Suppose $H$ is a spanning subgraph of $G$ so that $\widehat{\alpha}(I(M(H)))=\frac{|V(G)|}{2}$.  Then $\widehat{\alpha}(I(M(G))) = \frac{|V(G)|}{2}$.
\end{Corollary}
\begin{proof}
Since $H$ is a spanning subgraph of $G$, every spanning tree of $H$ is a spanning tree of $G$, so $I(M(H))\subseteq I(M(G))$.  Thus $\widehat{\alpha}(I(M(G)))\le \widehat{\alpha}(I(M(H))=\frac{|V(G)|}{2}$.  We also have $\frac{|V(G)|}{2}\le \widehat{\alpha}(I(M(G)))$ from \Cref{thm:grap-mat-wald-lb}, so the result follows.
\end{proof}

\subsection{Asymptotic Resurgence of Facet Ideals of Graphic Matroids} 
We obtain an upper bound for the asymptotic resurgence of facet ideals of graphic matroids and show that this bound is achieved by a larger class of graphic matroids.

\begin{Theorem}\label{thm:asym-gr-mat-up}
Let $G$ be a $2$-connected simple graph, and let $M = M(G)$ be its graphic matroid. Then
\[
\widehat{\rho}(I(M)) \leq 2 - \frac{2}{|V(G)|},
\]
with equality if and only if $\widehat{\alpha}(I(M))=\frac{|V(G)|}{2}$.  In particular, equality holds if $G$ is Hamiltonian, or if $G$ is regular with $\lambda(G) = \delta(G)$.
\end{Theorem}
\begin{proof}
We proceed by induction on the rank of $M$.  Note that, since $G$ is connected, $\rk(M)=|V(G)|-1$.  The only simple $2$-connected graph whose matroid has rank two is the triangle graph $T$.  The facet ideal $I=I(M(T))$ coincides with the ideal of three general points in $\mathbb{P}^2$.  It is well-known (for example \cite[Example 1.1]{DFMS19}) that $\widehat{\alpha}(I)=\frac{3}{2}=\frac{|V(G)|}{2}$ and $\widehat{\rho}(I)=\frac{4}{3}=2-\frac{2}{|V(G)|}$.  So the result holds for the matroid of $2$-connected simple graphs of rank $2$.

Now, suppose $\rk(M)>2$.  By \Cref{thm:asym-res-facet}, we have
\[
\widehat{\rho}(I(M))=\max\left\{ \frac{\rk(M)}{\widehat{\alpha}(I(M))},\ \widehat{\rho}_{c,1}(I(M)) \right\}.
\]
We first bound $\widehat{\rho}_{c,1}(I(M))$. For any edge $e \in E(G)$, $M/e$ is the graphic matroid of $G/e$. Although $G/e$ may contain parallel edges, let $\si(G/e)$ denote its simplification. Then $\si(G/e)$ is a simple graph on $|V(G)|-1$ vertices, and $M(\si(G/e)) \cong \si(M(G/e))$. By \Cref{thm:simplification}, 
$\widehat{\rho}(I(M(G/e))) = \widehat{\rho}(I(M(\si(G/e))))$.  Decompose $\si(G/e)$ into its $2$-connected blocks $H_1,\dots,H_t$.  Accordingly,
$M(\si(G/e)) \cong \bigoplus\limits_{i=1}^t M(H_i)$.  By \Cref{thm:direct-sum}, $\widehat{\rho}(I(M(G/e))) = \max\limits_{1 \leq i \leq t} \widehat{\rho}(I(M(H_i)))$. Applying induction for each $H_i$, we obtain $\widehat{\rho}(I(M(H_i))) \leq 2 - \frac{2}{|V(H_i)|} \leq 2 - \frac{2}{|V(G)|-1}$, and therefore, $\widehat{\rho}_{c,1}(I(M)) \leq 2 - \frac{2}{|V(G)|-1}<2-\frac{2}{|V(G)|}$.

Now, consider the ratio $\frac{\rk(M)}{\widehat{\alpha}(I(M))}$. By \Cref{thm:wald-gen-up}, $\widehat{\alpha}(I(M)) \geq \frac{|V(G)|}{2}$, and hence,  $\frac{\rk(M)}{\widehat{\alpha}(I(M))} \leq 2 - \frac{2}{|V(G)|},$ with equality if and only if $\widehat{\alpha}(I(M))=\frac{|V(G)|}{2}$.  This completes the induction.

If $G$ is Hamiltonian or if $G$ is regular with $\lambda(G)=\delta(G)$, then by \Cref{thm:grap-mat-wald-lb}, we have $\widehat{\alpha}(I(M)) = \frac{|V(G)|}{2}$, thus $\widehat{\rho}(I(M))=2 - \frac{2}{|V(G)|}$.
\end{proof}

\begin{Remark}\label{rem:AsymptoticComparison}
The upper bound in \Cref{thm:asym-gr-mat-up} always improves on the upper bound in \Cref{prop:asym-upper-bound}, and the improvement is significant if the number of vertices is large.  Moreover, the upper bound in \Cref{thm:asym-gr-mat-up} improves upon the upper bound in \Cref{thm:AsymptoticResurgenceBoundsMatroid} as long as $\cu(M(G)^*)\ge 2$, and the improvement is significant if $\cu(M(G)^*)$ is large.
\end{Remark}

We observe next that, for graphs, we can get a slight improvement on the lower bound of \Cref{thm:AsymptoticResurgenceBoundsMatroid}.

\begin{Corollary}\label{cor:asym-lb-grap}
Let $G$ be a simple $2$-connected graph.  Let $H$ be a subgraph of $G$ on $m$ vertices such that $\widehat{\alpha}(I(M(H)))=\frac{m}{2}$.  Then
$\widehat{\rho}(I(M(G))) \geq 2-\frac{2}{m}$.  If $H$ is also a spanning subgraph of $G$, then $\widehat{\rho}(I(M(G)))=2-\frac{2}{|V(G)|}$.
\end{Corollary}
\begin{proof}
By \Cref{thm:asym-gr-mat-up}, $\widehat{\rho}(I(M(H)))=2-\frac{2}{m}$.  Since $H$ is a subgraph of $G$, the graphic matroid $M(H)$ is a minor of $M(G)$.  By \Cref{prop:minors}, we have $\widehat{\rho}(I(M(H))) \leq \widehat{\rho}(I(M(G)))$.    Therefore,
$\widehat{\rho}(I(M(G))) \geq 2-\frac{2}{m}$.  If $H$ is a spanning subgraph, then $m=|V(G)|$ and we also have $\widehat{\rho}(I(M(G))) = 2-\frac{2}{m}$ by \Cref{thm:asym-gr-mat-up}.
\end{proof}

\begin{Remark}\label{rem:AddingEdgesToCycle}
Observe that, by \Cref{cor:asym-lb-grap}, if $C$ is a cycle on $n$ vertices and $G$ is a graph formed from $C$ by adding in any number of edges between the vertices of $C$, then $\widehat{\rho}(I(M(G)))=\widehat{\rho}(I(M(C))=2-\frac{2}{n}$.
\end{Remark}

\section{Examples}\label{sec:examples}
In this section, we examine some families of examples of graphs. If $I(M(G))$ is the facet ideal of a graphic matroid associated to a graph $G$ that is either Hamiltonian or regular, then the Waldschmidt constant attains the lower bound in \Cref{thm:grap-mat-wald-lb}. Consequently, in these cases, the asymptotic resurgence attains the upper bound as in \Cref{thm:asym-gr-mat-up}. The following example shows that neither of these conditions on $G$ is necessary to obtain equality for the Waldschmidt constant.

\begin{Example}\label{ex:nonham-nonregular-lower-bound-sharp}
Let $G$ be the graph shown in \Cref{fig:nonham-nonregular-nine-vertices}. Observe that $G$ is a simple planar graph with $|V(G)|=9$. Let $M=M(G)$. Note that $G$ is neither Hamiltonian nor regular. Nonetheless, we show that the bounds in \Cref{thm:grap-mat-wald-lb} and \Cref{thm:asym-gr-mat-up} are attained, that is $\widehat{\alpha}(I(M))=\frac{|V(G)|}{2}$ and $\widehat{\rho}(I(M))=2-\frac{2}{|V(G)|}$.

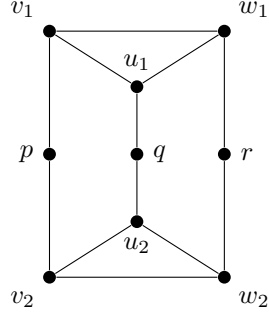
\begin{figure}[ht]
\centering
\begin{tikzpicture}[scale=1.05,
    vertex/.style={circle, fill=black, inner sep=1.7pt},
    every edge/.style={draw, line width=0.7pt}
]
\node[vertex,label=left:$p$] (p) at (-1.1,0) {};
\node[vertex,label=right:$q$] (q) at (0,0) {};
\node[vertex,label=right:$r$] (r) at (1.1,0) {};

\node[vertex,label=above left:$v_1$] (v1) at (-1.1,1.55) {};
\node[vertex,label=above right:$w_1$] (w1) at (1.1,1.55) {};
\node[vertex,label=above:$u_1$] (u1) at (0,0.85) {};

\node[vertex,label=below left:$v_2$] (v2) at (-1.1,-1.55) {};
\node[vertex,label=below right:$w_2$] (w2) at (1.1,-1.55) {};
\node[vertex,label=below:$u_2$] (u2) at (0,-0.85) {};

\foreach \i in {1,2}{
    \draw (u\i)--(v\i);
    \draw (u\i)--(w\i);
    \draw (v\i)--(w\i);
}

\draw (p)--(v1);
\draw (p)--(v2);
\draw (q)--(u1);
\draw (q)--(u2);
\draw (r)--(w1);
\draw (r)--(w2);

\end{tikzpicture}
\caption{A simple planar, non-Hamiltonian, non-regular graph $G$ on $9$ vertices.}
\label{fig:nonham-nonregular-nine-vertices}
\end{figure}

By \Cref{thm:grap-mat-wald-lb},
$\widehat{\alpha}(I(M)\geq \frac{|V(G)|}{2}=\frac{9}{2}$. To prove the reverse inequality, define
$\mathbf a\in \mathbb R_{\geq 0}^{|E(G)|}$ as
\[
a_e=
\begin{cases}
\frac14, & \text{if } e\in \{u_iv_i,u_iw_i,v_iw_i : i=1,2\},\\[2mm]
\frac12, & \text{otherwise}.
\end{cases}
\]

We show that $\mathbf a\in \SP(I(M))$. The defining
inequalities of $\SP(I(M))$ are $\sum\limits_{e \in B} a_e \geq 1$ for every bond $B$ of $G$. If $|B| \geq 4$, then $\sum\limits_{e \in B} a_e \geq  4\cdot \frac{1}{4}=1$. The only bonds of cardinality two are $\{pv_1,pv_2\},$ $\{qu_1,qu_2\}, $ and $\{rw_1,rw_2\}$
and for each of them, we have  $\sum\limits_{e \in B} a_e =1$.
Furthermore, every bond $B$ of cardinality three contains at least one edge incident to one of the vertices $p,q,r$, and every such edge has weight $\frac{1}{2}$. Since the remaining two edges have weight at least $\frac{1}{4}$, it follows that
$\sum\limits_{e\in B} a_e\geq 1$. Hence, $\mathbf a\in \SP(I(M))$.

Thus, $\widehat{\alpha}(I(M))\leq \sum\limits_{e\in E(G)} a_e= \frac{9}{2},$ and hence, $\widehat{\alpha}(I(M))= \frac{|V(G)|}{2}$.  By \Cref{thm:asym-gr-mat-up}, $\widehat{\rho}(I(M))=2-\frac{2}{9}=\frac{16}{9}$.
\end{Example}

\begin{Remark}
    Computations in Sage~\cite{sagemath} (see the entry at the bottom of the rightmost column in \Cref{tab:enumeration}) indicate that there are  only four non-Hamiltonian graphs on 9 vertices up to isomorphism satisfying $\widehat{\alpha}(I(M(H)))=9/2$.  These are the graph $G$ in \Cref{ex:nonham-nonregular-lower-bound-sharp} and the three graphs obtained from $G$ by adding in some subset of the edges $\{v_1,v_2\},\{u_1,u_2\}, \mbox{ and } \{w_1,w_2\}$.  Thus the graph $G$ in \Cref{ex:nonham-nonregular-lower-bound-sharp} is the unique minimal non-Hamiltonian graph on $9$ vertices satisfying $\widehat{\alpha}(I(M(G)))=9/2$.
    \end{Remark}

In our remaining examples, we consider graphs $G$ so that the asymptotic resurgence of the facet ideal of $G$ achieves the lower bound of \Cref{thm:asym-gr-mat-up}.  Moreover, in each of these examples, $\widehat{\rho}(I(G))>\frac{\alpha(I(G))}{\widehat{\alpha}(I(G))}$.  That is, we have strict inequality in \GHV{}.

A \emph{Theta graph}, $\Theta_{r,s,t}$, is the union of three internally disjoint (simple) paths of lengths $r,s,t\ge 2$ that have the same two distinct end vertices. Equivalently, one obtains $\Theta_{r,s,t}$ by taking 3 paths of lengths $r,s,t$ and gluing them together at their endpoints.  

\begin{Theorem}\label{theta-graphs}
Let $G=\Theta_{r,s,t}$ be a  Theta graph with $r\le s\le t$ and let $M=M(G)$. Then 
\begin{enumerate}
\item $\widehat{\alpha}(I(M))=\frac{r+s+t}{2}$ and
\item $\widehat{\rho}(I(M))=2-\frac{2}{\cu(G)}=2-\frac{2}{s+t}$.
\end{enumerate}
In particular, $\widehat{\rho}(I(M))>\frac{\alpha(I(M))}{\widehat{\alpha}(I(M))}$.
\end{Theorem}

\begin{proof}
First notice that $|E(G)|=r+s+t$, and thus, by \Cref{thm:wald-gen-up}~(a) $\widehat{\alpha}(I(M))\le \frac{r+s+t}{2}$, since $\lambda(G)=2$.

Let $P_r, P_s, P_t$ be the paths of length $r,s,t$, respectively in $G$.   Any two edges of $P_r$ form a bond of $G$. Let $C\in \C(M^*)$ be formed by taking any two edges on $P_r$. Then $\chi_C \cdot {\bf a}=\sum\limits_{e\in C}a_e\ge 1$ for all ${\bf a} \in \SP(I(M))$. Now, there are $\binom{r}{2}$ choices of such $C$, and adding all these inequalities together,  we obtain
\[ (r-1)\sum\limits_{e\in E(P_r)} a_e\ge \binom{r}{2}.\] Thus, $\sum \limits_{e\in E(P_r)} a_e\ge \frac{r}{2}$.  Similarly, $\sum\limits_{e \in E(P_s)} a_e\ge \frac{s}{2}$ and $\sum\limits_{e \in E(P_t)} a_e\ge \frac{t}{2}$.  Hence, $\sum\limits_{e\in E(G)} a_e\ge \frac{r+s+t}{2}$ for all ${\bf a} \in \SP(I(M))$, and so, by \Cref{thm:SPWaldschmidt},  $\widehat{\alpha}(I(M))\ge \frac{r+s+t}{2}$, completing the proof of (a).

We prove (b) by induction on $|E(G)|=r+s+t$.  Since $r,s,t\ge 2$, our base case is $r=s=t=2$.  In this case, for any $e \in E(G)$, $G/e$ is a four-cycle with a chord, which is Hamiltonian, and thus $\widehat{\rho}_{c,1}(I(M))=2-\frac{2}{4}=\frac{3}{2}$ by \Cref{thm:asym-gr-mat-up}.  By part~(a), $\widehat{\alpha}(I(M))=3$.  Thus $\widehat{\rho}(I(M))=\max\{4/3,3/2\}=3/2$, by \Cref{thm:asym-res-facet}.

 Now suppose that $r+s+t>6$.  First we claim that $\widehat{\rho}_{c,1}(I(M))=2-\frac{2}{s+t}$. We may assume that $t>2$.  If $r=2$, then contracting an edge of $P_r$ results in a Hamiltonian graph on $r+s+t-2=s+t$ vertices. Hence, the asymptotic resurgence of this contraction is $2-\frac{2}{s+t}$ by \Cref{thm:asym-gr-mat-up}.  On the other hand, since $\Theta_{2,s,t}$ has $s+t+1$ vertices, contracting any edge $e$ results in a simple $2$-connected graph $G$ on $s+t$ vertices.  By \Cref{thm:asym-gr-mat-up}, $\widehat{\rho}(I(M/e))\le 2-\frac{2}{s+t}$.  It follows that if $r=2$, $\widehat{\rho}_{c,1}(I(M))=2-\frac{2}{s+t}$.

     Finally, if $r>2$, that is if all the paths have length greater than $2$, then contracting any edge along $P_{r}, P_{s}, P_{t}$ results in the matroids $\Theta_{r-1,s,t}$, $\Theta_{r,s-1,t}$, or $\Theta_{r,s,t-1}$, respectively. Thus, by induction, $\widehat{\rho}_{c,1}(I(M))=\max\{2-\frac{2}{\cu(\Theta_{r-1,s,t})}, 2-\frac{2}{\cu(\Theta_{r,s-1,t})}, 2-\frac{2}{\cu(\Theta_{r,s,t-1})}\}=2-\frac{2}{s+t}$, as claimed. 
     The result now follows by part (a) and \Cref{thm:asym-res-facet}, since $\frac{\alpha(I(M))}{\widehat{\alpha}(I(M))}= \frac{2(r+s+t-2)}{r+s+t}=2-\frac{4}{r+s+t}<2-\frac{2}{s+t}$.
     \end{proof}

Let $m,n \geq 2$. The \emph{Jahangir graph} $J_{m,n}$ is obtained as follows. Let $C_{mn}$ be a cycle with vertex set $V(C_{mn})=\{v_1,v_2,\dots,v_{mn}\}$
listed in cyclic order. Add a new vertex $z$, and join $z$ to the vertices $v_1, v_{1+m}, v_{1+2m}, \dots, v_{1+(n-1)m}$. We call $z$ the
\textit{center}, the edges of $C_{mn}$ the \textit{rim edges}, and the edges
incident to $z$ the \textit{spokes}. Thus $|V(J_{m,n})|=mn+1$.

\begin{Theorem}\label{jahangir-graphs}
Let $m,n \ge 2$, $G=J_{m,n}$, and let $M=M(G)$. Then \begin{enumerate}
\item $\widehat{\alpha}(I(M))=\frac{mn+2}{2}$ and
\item  $\widehat{\rho}(I(M))=2-\frac{2}{mn}=2-\frac{2}{\cu(G)}$.
\end{enumerate}
In particular, $\widehat{\rho}(I(M))>\frac{\alpha(I(M))}{\widehat{\alpha}(I(M))}$.
\end{Theorem}

\begin{proof}
We first compute $\widehat{\alpha}(I(M))$. By \Cref{cor:SPIM}, if $\mathbf{a}=(a_e)_{e\in E(G)}\in \SP(I(M))$, then $\sum\limits_{e\in B} a_e \geq 1$ for all bonds $B$ of $G$.

Let $\mathbf{a}\in \SP(I(M))$.  Consider one of the $n$ rim paths between two consecutive neighbors of $z$, and let its edges be $e_1,\ldots,e_m$. For every $1\leq i<j\leq m$, $\{e_i,e_j\}$ is a bond of $G$, and therefore, $a_{e_i}+a_{e_j}\geq 1$. By adding these
$\binom{m}{2}$ inequalities, we get
$(m-1)(a_{e_1}+\cdots+a_{e_m})\geq \binom{m}{2}$ and hence $a_{e_1}+\cdots+a_{e_m}\geq \frac{m}{2}$.
 Now, by adding this inequality over all $n$ such rim paths, we get $\sum\limits_{e\in E(C_{mn})} a_e \geq \frac{mn}{2}$. Also, the set of spokes is a vertex bond, so $\sum\limits_{z \in e} a_e\geq 1$. Hence, for every $\mathbf {a} \in \SP(I(M))$,
$\sum\limits_{e\in E(G)} a_e \geq \frac{mn}{2}+1=\frac{mn+2}{2}$. Thus, by \Cref{thm:SPWaldschmidt}, $\widehat{\alpha}(I(M))\geq \frac{mn+2}{2}$. The reverse inequality follows from \Cref{thm:grap-mat-wald-lb}~(a). 
Hence, 
$\widehat{\alpha}(I(M)) = \frac{mn+2}{2}$.

We now compute $\widehat{\rho}(I(M))$. Since $G$ is simple $2$-connected, $\rk(M)=|V(G)|-1=mn$. By \Cref{thm:asym-res-facet}, we have
$\widehat{\rho}(I(M))=\max\left\{
\frac{\rk(M)}{\widehat{\alpha}(I(M))},
\widehat{\rho}_{c,1}(I(M))
\right\}$.
Using the first part,
$\frac{\rk(M)}{\widehat{\alpha}(I(M))}=\frac{2mn}{mn+2}$. By definition,
$\widehat{\rho}_{c,1}(I(M)) = \max\limits_{e \in E(G)} \widehat{\rho}(I(M/e))$.
Since $M = M(G)$ is a graphic matroid, we have $M/e \cong M(G/e)$ for every $e \in E(G)$, and hence,
$\widehat{\rho}_{c,1}(I(M)) = \max\limits_{e \in E(G)} \widehat{\rho}(I(M(G/e)))$.

If $e$ is a spoke, then $G/e$ is a Hamiltonian graph on $mn$ vertices. By \Cref{thm:asym-gr-mat-up}, $\widehat{\rho}(I(M(G/e))) = 2 - \frac{2}{mn}$, and hence $\widehat{\rho}_{c,1}(I(M)) \geq 2 - \frac{2}{mn}$. On the other hand, \Cref{thm:asym-gr-mat-up} implies that for every $e \in E(G)$,
$\widehat{\rho}(I(M(G/e))) \leq 2 - \frac{2}{mn}$.
Thus, $\widehat{\rho}_{c,1}(I(M)) = 2 - \frac{2}{mn}$.
Finally, since $2 - \frac{2}{mn} > \frac{2mn}{mn+2}$ for $mn > 2$, it follows that
$\widehat{\rho}(I(M)) = 2 - \frac{2}{mn}$.\end{proof}

\begin{Example}\label{ex:cubic-planar-lambda-two}
Let $G$ be the graph shown in \Cref{fig:cubic-planar-lambda-two}. Notice that $G$ is a simple planar $3$-regular graph with $|V(G)|=14$, $|E(G)|=21$, and $\lambda(G)=2$.

\begin{figure}[ht]
\centering
\begin{tikzpicture}[scale=1.05,
    vertex/.style={circle, fill=black, inner sep=1.7pt},
    every edge/.style={draw, line width=0.7pt}
]
\node[vertex,label=left:$p$] (p) at (-3,0) {};
\node[vertex,label=right:$q$] (q) at (3,0) {};

\node[vertex,label=left:$y_1$] (y1) at (-1.2,2.4) {};
\node[vertex,label=right:$z_1$] (z1) at (1.2,2.4) {};
\node[vertex,label=above:$u_1$] (u1) at (0,3.15) {};
\node[vertex,label=below:$v_1$] (v1) at (0,1.65) {};

\node[vertex,label=above:$y_2$] (y2) at (-1.2,0) {};
\node[vertex,label=above:$z_2$] (z2) at (1.2,0) {};
\node[vertex,label=above:$u_2$] (u2) at (0,0.75) {};
\node[vertex,label=below:$v_2$] (v2) at (0,-0.75) {};

\node[vertex,label=left:$y_3$] (y3) at (-1.2,-2.4) {};
\node[vertex,label=right:$z_3$] (z3) at (1.2,-2.4) {};
\node[vertex,label=above:$u_3$] (u3) at (0,-1.65) {};
\node[vertex,label=below:$v_3$] (v3) at (0,-3.15) {};

\foreach \i in {1,2,3}{
    \draw (y\i)--(u\i);
    \draw (u\i)--(z\i);
    \draw (z\i)--(v\i);
    \draw (v\i)--(y\i);
    \draw (u\i)--(v\i);
}

\foreach \i in {1,2,3}{
    \draw (p)--(y\i);
    \draw (q)--(z\i);
}
\end{tikzpicture}
\caption{A simple planar $3$-regular graph $G$ with edge-connectivity $2$.}
\label{fig:cubic-planar-lambda-two}
\end{figure}
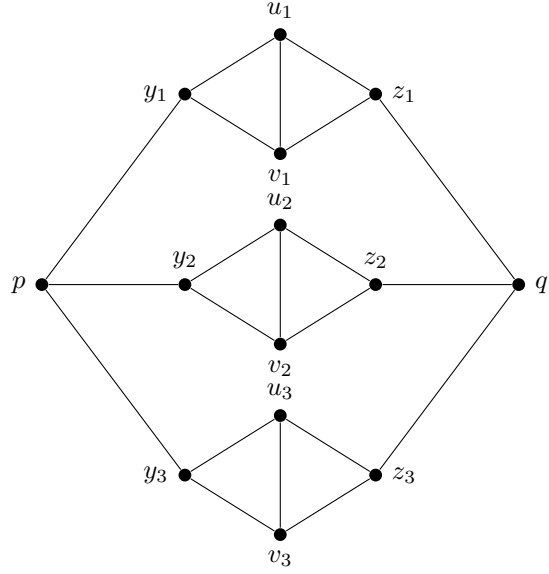

We claim that $ \widehat{\alpha}(I(M(G)))=\frac{15}{2}. $ For each vertex $w\in V(G)\setminus\{p,q\}$, let $ B_w=\{e\in E(G): w\in e\}, $ and for $i=1,2,3$, let $ B_i=\{py_i,qz_i\}. $ The collection $ \mathcal{D} = \{B_w : w\in V(G)\setminus\{p,q\}\} \cup \{B_1,B_2,B_3\} $ consists of fifteen bonds. Moreover, every edge of $G$ belongs to exactly two members of $\mathcal{D}$. Summing the corresponding symbolic polyhedron inequalities yields $ 2\sum\limits_{e\in E(G)}a_e\ge 15, $ and therefore, $ \widehat{\alpha}(I(M(G)))\ge \frac{15}{2}. $

To prove the reverse inequality, define $\mathbf a\in\mathbb R_{\ge 0}^{|E(G)|}$ by \[ a_e= \begin{cases} \frac12, & e\in \{py_i,qz_i,u_iv_i : i=1,2,3\},\\[2mm] \frac14, & \text{otherwise}. \end{cases} \]
Let $B$ be a bond of $G$. If $|B|\geq 4$, then $\sum\limits_{e\in B}a_e\geq 4\cdot \frac{1}{4}=1$. The only bonds of size two are
$\{py_i,qz_i\},$ for $i=1,2,3$. Moreover, every bond of size three contains at least one edge from $\{py_i,qz_i,u_iv_i : i=1,2,3\}$. It follows that
$\sum\limits_{e\in B} a_e \ge 1
$ for every bond $B$ of $G$.  Therefore, $\mathbf a\in \SP(I(M(G)))$.

Since
$\sum\limits_{e\in E(G)} a_e=
\frac{15}{2},$ we obtain
$\widehat{\alpha}(I(M(G)))\le \frac{15}{2}$. Hence, $
\widehat{\alpha}(I(M(G)))=\frac{15}{2} >\frac{|V(G)|}{2}$.

A Sage~\cite{sagemath} computation based on the approach described in \Cref{rmk:asym-res-algo} shows that
\[
\widehat{\rho}(I(M(G)))=\frac{9}{5}=2-\frac{2}{\cu(G)}.
\]
In particular, $\widehat{\rho}(I(M(G))>\frac{\alpha(I(M(G)))}{\widehat{\alpha}(I(M(G)))}$.
\end{Example}

\section{Concluding remarks and questions}\label{sec:Conclusion}

In \Cref{sec:reduction-techniques}, we reduced the computation of the asymptotic resurgence of facet ideals of matroids to those matroids that are simple and connected.  Thus, we restrict our concluding discussion to simple and $2$-connected graphs.  It follows from  \Cref{thm:AsymptoticResurgenceBoundsMatroid} and  \Cref{thm:asym-gr-mat-up} that if $G$ is a $2$-connected simple graph, then $2-\frac{2}{\cu(G)} \leq \widehat{\rho}(I(M(G))) \leq 2-\frac{2}{|V(G)|}$. Moreover, the upper bound is attained if $G$ is Hamiltonian.  This naturally raises the following two questions.

\begin{Question}
What graph-theoretic properties characterize those non-Hamiltonian  graphs $G$ for which $\widehat{\rho}(I(M(G)))=2-\frac{2}{|V(G)|}$?  Equivalently (by \Cref{thm:asym-gr-mat-up}), what graph-theoretic properties characterize non-Hamiltonian graphs $G$ satisfying
$\widehat{\alpha}(I(M(G)))=\frac{|V(G)|}{2}$?
\end{Question}

\begin{Question}\label{ques:twobounds}
Suppose $G$ is a simple $2$-connected graph.
Is the asymptotic resurgence of $I(M(G))$ always either $2-\frac{2}{\cu(G)}$ or $2-\frac{2}{|V(G)|}$?
\end{Question}

We expect that the answer to \Cref{ques:twobounds} is no, but computer computations in Sage~\cite{sagemath} -- which we record in \Cref{tab:enumeration} -- indicate that the answer is yes for every $2$-connected graph on nine vertices or less.  See \Cref{rmk:asym-res-algo} for a brief discussion of how we performed these computations.

\begin{table}[ht]
\centering
\renewcommand{\arraystretch}{1.3}
\setlength{\tabcolsep}{7pt}

\begin{tabular}{
    |c|
    r|
    r|
    r|
    >{\centering\arraybackslash}p{3.7cm}|
    >{\centering\arraybackslash}p{3.7cm}|
}
\hline
\multirow{4}{*}{$n$}
&
\multicolumn{5}{c|}{Number of nonisomorphic graphs on $n$ vertices}
\\
\cline{2-6}
&
&
\multicolumn{4}{c|}{Number of $2$-connected graphs}
\\
\cline{3-6}
&
&
&
\multicolumn{3}{c|}{Number of non-Hamiltonian graphs}
\\
\cline{4-6}
&
\multicolumn{1}{c|}{Total}
&
\multicolumn{1}{c|}{Total}
&
\multicolumn{1}{c|}{Total}
&
\multicolumn{1}{c|}{
    \begin{tabular}{c}
    Satisfying\\[1mm]
    $\widehat{\rho}(I(M(G)))
      =2-\dfrac{2}{\cu(G)}$
    \end{tabular}
}
&
\multicolumn{1}{c|}{
    \begin{tabular}{c}
    Satisfying\\[1mm]
    $\widehat{\rho}(I(M(G)))
      =2-\dfrac{2}{|V(G)|}$
    \end{tabular}
}
\\
\hline
4 &     11 &      3 &     0 &     0 & 0 \\
\hline
5 &     34 &     10 &     2 &     2 & 0 \\
\hline
6 &    156 &     56 &     8 &     8 & 0 \\
\hline
7 &   1044 &    468 &    85 &    85 & 0 \\
\hline
8 &  12346 &   7123 &   927 &   927 & 0 \\
\hline
9 & 274668 & 194066 & 16983 & 16979 & 4 \\
\hline
\end{tabular}
\caption{Number of nonisomorphic graphs on $n$ vertices attaining the bounds in \Cref{thm:AsymptoticResurgenceBoundsMatroid} and \Cref{thm:asym-gr-mat-up}.}
\label{tab:enumeration}
\end{table}

\begin{Remark}\label{rmk:asym-res-algo}
As is discussed in \cite[Appendix~A]{DFK26}, an algorithm to compute the asymptotic resurgence of a squarefree monomial ideal $I$ is implicit from~\cite{DFMS19}.  Namely, let $\nu_1,\ldots,\nu_k$ be the linear functionals corresponding to inward pointing normal vectors of the Newton polyhedron $\NP(I)$.  Put $\nu_i(I)=\min\{\nu_i(\mathbf{a}):\mathbf{a}\in \NP(I)\}$ and  $\widehat{\nu}_i(I)=\min\{\nu_i(\mathbf{a}):\mathbf{a}\in \SP(I)\}$ for $i=1,\ldots,k$.  Then $\widehat{\rho}(I)=\min\left\lbrace \frac{\nu_i(I)}{\widehat{\nu}_i(I)}: i=1,\ldots,k\right\rbrace$.

In case $M$ is a matroid and $I(M)$ is its facet ideal, the situation is much improved since the inequalities defining $\NP(I(M))$ come from the well-known inequalities for matroid polytopes (see \cite[Remark~3.10]{DFK26}).  Namely, the defining inequalities of $\NP(I(M))$ are all of the form $\chi_{E\setminus F}\cdot \mathbf{a}\ge \rk(M)-\rk_M(F)$, where $F$ is a flat of $M$.  This observation allows one to compute the asymptotic resurgence provided that one can compute the flats of a matroid and minimize the functionals $\chi_{E\setminus F}$ over the symbolic polyhedron of $I(M)$.  This is possible in Macaulay2, however, we found that implementations in Sage run considerably faster by taking advantage of Python's linear programming routines.  The code we used to compute asymptotic resurgence can be found under the research tab on the website of the first author (\href{midipasq.github.io}{midipasq.github.io}).
\end{Remark}

We end with three more questions.

\begin{Question}
What effect do graph-theoretic operations like series extension 
have on the Waldschmidt constant and asymptotic resurgence of $I(M(G))$?
\end{Question}

\begin{Question}
If $I(G)$ is the edge ideal of a graph $G$, then $\widehat{\alpha}(I(G))$ is the \textit{fractional chromatic number} of $G$~\cite{BocciCooperGuardoHarbourneJanssenNagelSeceleanuVanTuylVu}.  Is there a correspondingly `nice' interpretation of $\widehat{\alpha}(I(M(G)))$ in terms of a graph-theoretic invariant?
\end{Question}

Following up on
\Cref{thm:AsymptoticResurgenceBoundsMatroid} and \Cref{rem:refinement} we ask the following question.
\begin{Question}\label{ques:asymptoticresbound}
If $I$ is a squarefree monomial ideal of codimension $c$ and $\omega$ is the largest degree of a generator of $I$, is it always true that
$
\widehat{\rho}(I)\le \dfrac{c\omega}{c+\omega-1}
$?
\end{Question}
We suspect the answer to \Cref{ques:asymptoticresbound} may be no, but we have not found a counterexample.

\section{Acknowledgments}
DiPasquale was partially supported by NSF grant DMS--2344588. Kumar was partially supported by an AMS-Simons Travel grant.  

All results, proofs, and mathematical content were produced without AI assistance.  The authors used Google Gemini to obtain a starting point for the code to compute asymptotic resurgence (see \Cref{rmk:asym-res-algo}).  The authors verified and improved upon this code, which is now posted under the research tab of DiPasquale's website (\href{midipasq.github.io}{midipasq.github.io}).

\bibliography{bibl}

@incollection {MG18,
    AUTHOR = {Lampa-Baczy\'nska, Magdalena and Malara, Grzegorz},
     TITLE = {On the containment hierarchy for simplicial ideals},
 BOOKTITLE = {Extended abstracts {F}ebruary 2016---positivity and
              valuations},
    SERIES = {Trends Math. Res. Perspect. CRM Barc.},
    VOLUME = {9},
     PAGES = {71--74},
 PUBLISHER = {Birkh\"auser/Springer, Cham},
      YEAR = {2018},
      ISBN = {978-3-030-00027-1; 978-3-030-00026-4},
   MRCLASS = {13F20},
  MRNUMBER = {3946256},
       DOI = {10.1007/978-3-030-00027-1\_11},
       URL = {https://doi.org/10.1007/978-3-030-00027-1_11},
}

@article {ELS01,
    AUTHOR = {Ein, L. and Lazarsfeld, R. and Smith, K. E.},
     TITLE = {Uniform bounds and symbolic powers on smooth varieties},
   JOURNAL = {Invent. Math.},
  FJOURNAL = {Inventiones Mathematicae},
    VOLUME = {144},
      YEAR = {2001},
    NUMBER = {2},
     PAGES = {241--252},
      ISSN = {0020-9910,1432-1297},
   MRCLASS = {13A10 (13H05 14Q20)},
  MRNUMBER = {1826369},
MRREVIEWER = {Irena\ Swanson},
       DOI = {10.1007/s002220100121},
       URL = {https://doi.org/10.1007/s002220100121},
}

@article {HH02,
    AUTHOR = {Hochster, M. and Huneke, C.},
     TITLE = {Comparison of symbolic and ordinary powers of ideals},
   JOURNAL = {Invent. Math.},
  FJOURNAL = {Inventiones Mathematicae},
    VOLUME = {147},
      YEAR = {2002},
    NUMBER = {2},
     PAGES = {349--369},
      ISSN = {0020-9910,1432-1297},
   MRCLASS = {13A10 (13H05)},
  MRNUMBER = {1881923},
MRREVIEWER = {Irena\ Swanson},
       DOI = {10.1007/s002220100176},
       URL = {https://doi.org/10.1007/s002220100176},
}

@article {MS17,
    AUTHOR = {Ma, Linquan and Schwede, Karl},
     TITLE = {Perfectoid multiplier/test ideals in regular rings and bounds
              on symbolic powers},
   JOURNAL = {Invent. Math.},
  FJOURNAL = {Inventiones Mathematicae},
    VOLUME = {214},
      YEAR = {2018},
    NUMBER = {2},
     PAGES = {913--955},
      ISSN = {0020-9910,1432-1297},
   MRCLASS = {13A35 (14F18)},
  MRNUMBER = {3867632},
MRREVIEWER = {Ana\ Bravo},
       DOI = {10.1007/s00222-018-0813-1},
       URL = {https://doi.org/10.1007/s00222-018-0813-1},
}

@article {DiPasquale-Drabkin-2021,
    AUTHOR = {DiPasquale, Michael and Drabkin, Ben},
     TITLE = {On resurgence via asymptotic resurgence},
   JOURNAL = {J. Algebra},
  FJOURNAL = {Journal of Algebra},
    VOLUME = {587},
      YEAR = {2021},
     PAGES = {64--84},
      ISSN = {0021-8693,1090-266X},
   MRCLASS = {13A30 (13A15 13F20)},
  MRNUMBER = {4301520},
MRREVIEWER = {Kriti\ Goel},
       DOI = {10.1016/j.jalgebra.2021.07.021},
       URL = {https://doi.org/10.1016/j.jalgebra.2021.07.021},
}

@article {Villarreal-2023,
    AUTHOR = {Villarreal, Rafael H.},
     TITLE = {A duality theorem for the ic-resurgence of edge ideals},
   JOURNAL = {European J. Combin.},
  FJOURNAL = {European Journal of Combinatorics},
    VOLUME = {109},
      YEAR = {2023},
     PAGES = {Paper No. 103656, 18},
      ISSN = {0195-6698,1095-9971},
   MRCLASS = {90C90 (05E40 13A70 52B12 90C05 90C46)},
  MRNUMBER = {4522421},
MRREVIEWER = {Mehrdad\ Nasernejad},
       DOI = {10.1016/j.ejc.2022.103656},
       URL = {https://doi.org/10.1016/j.ejc.2022.103656},
}

@book {Wel,
    AUTHOR = {Welsh, D. J. A.},
     TITLE = {Matroid theory},
    SERIES = {L. M. S. Monographs},
    VOLUME = {No. 8},
 PUBLISHER = {Academic Press [Harcourt Brace Jovanovich, Publishers],
              London-New York},
      YEAR = {1976},
     PAGES = {xi+433},
   MRCLASS = {05B35},
  MRNUMBER = {427112},
MRREVIEWER = {W.\ T.\ Tutte},
}

@article {JKM22,
    AUTHOR = {Jayanthan, A. V. and Kumar, Arvind and Mukundan, Vivek},
     TITLE = {On the resurgence and asymptotic resurgence of homogeneous
              ideals},
   JOURNAL = {Math. Z.},
  FJOURNAL = {Mathematische Zeitschrift},
    VOLUME = {302},
      YEAR = {2022},
    NUMBER = {4},
     PAGES = {2407--2434},
      ISSN = {0025-5874,1432-1823},
   MRCLASS = {13F55 (05E40 13A15 13F20)},
  MRNUMBER = {4509032},
MRREVIEWER = {Somayeh\ Bandari},
       DOI = {10.1007/s00209-022-03138-w},
       URL = {https://doi.org/10.1007/s00209-022-03138-w},
}

@article {CEHH17,
    AUTHOR = {Cooper, S. M. and Embree, R. J. D. and H\`a, T. H. and Hoefel, A. H.},
     TITLE = {Symbolic powers of monomial ideals},
   JOURNAL = {Proc. Edinb. Math. Soc. (2)},
  FJOURNAL = {Proceedings of the Edinburgh Mathematical Society. Series II},
    VOLUME = {60},
      YEAR = {2017},
    NUMBER = {1},
     PAGES = {39--55},
      ISSN = {0013-0915,1464-3839},
   MRCLASS = {13F20 (13A02 14N05)},
  MRNUMBER = {3589840},
MRREVIEWER = {Mike\ Janssen},
       DOI = {10.1017/S0013091516000110},
       URL = {https://doi.org/10.1017/S0013091516000110},
}

@article {DFMS19,
    AUTHOR = {DiPasquale, M. and Francisco, C. A. and Mermin, J. and Schweig, J.},
     TITLE = {Asymptotic resurgence via integral closures},
   JOURNAL = {Trans. Amer. Math. Soc.},
  FJOURNAL = {Transactions of the American Mathematical Society},
    VOLUME = {372},
      YEAR = {2019},
    NUMBER = {9},
     PAGES = {6655--6676},
      ISSN = {0002-9947,1088-6850},
   MRCLASS = {14C20 (13A18 13B22 13F20)},
  MRNUMBER = {4024534},
MRREVIEWER = {Carlos\ Galindo},
       DOI = {10.1090/tran/7835},
       URL = {https://doi.org/10.1090/tran/7835},
}

@article {BB83,
    AUTHOR = {Bollob\'as, B\'ela},
     TITLE = {Almost all regular graphs are {H}amiltonian},
   JOURNAL = {European J. Combin.},
  FJOURNAL = {European Journal of Combinatorics},
    VOLUME = {4},
      YEAR = {1983},
    NUMBER = {2},
     PAGES = {97--106},
      ISSN = {0195-6698,1095-9971},
   MRCLASS = {05C45},
  MRNUMBER = {705962},
MRREVIEWER = {Edward\ A.\ Bender},
       DOI = {10.1016/S0195-6698(83)80039-0},
       URL = {https://doi.org/10.1016/S0195-6698(83)80039-0},
}

@article{KM2026,
  author  = {Kumar, Arvind and Mukundan, Vivek},
  title   = {Symbolic powers of classical varieties},
  journal = {Israel Journal of Mathematics},
  year    = {2026},
  doi     = {10.1007/s11856-026-2931-6},
  url     = {https://doi.org/10.1007/s11856-026-2931-6},
  issn    = {1565-8511}
}

@book {Oxley2011,
    AUTHOR = {Oxley, J.},
     TITLE = {Matroid theory},
    SERIES = {Oxford Graduate Texts in Mathematics},
    VOLUME = {21},
   EDITION = {Second},
 PUBLISHER = {Oxford University Press, Oxford},
      YEAR = {2011},
     PAGES = {xiv+684},
      ISBN = {978-0-19-960339-8},
   MRCLASS = {05-01 (05B35 90C27)},
  MRNUMBER = {2849819},
MRREVIEWER = {Maruti\ M.\ Shikare},
       DOI = {10.1093/acprof:oso/9780198566946.001.0001},
       URL = {https://doi.org/10.1093/acprof:oso/9780198566946.001.0001},
}

@article {BH10,
    AUTHOR = {Bocci, Cristiano and Harbourne, Brian},
     TITLE = {Comparing powers and symbolic powers of ideals},
   JOURNAL = {J. Algebraic Geom.},
  FJOURNAL = {Journal of Algebraic Geometry},
    VOLUME = {19},
      YEAR = {2010},
    NUMBER = {3},
     PAGES = {399--417},
      ISSN = {1056-3911,1534-7486},
   MRCLASS = {13F20 (13A15)},
  MRNUMBER = {2629595},
MRREVIEWER = {Irena\ Swanson},
       DOI = {10.1090/S1056-3911-09-00530-X},
       URL = {https://doi.org/10.1090/S1056-3911-09-00530-X},
}

@article {GHV13,
    AUTHOR = {Guardo, Elena and Harbourne, Brian and Van Tuyl, Adam},
     TITLE = {Asymptotic resurgences for ideals of positive dimensional
              subschemes of projective space},
   JOURNAL = {Adv. Math.},
  FJOURNAL = {Advances in Mathematics},
    VOLUME = {246},
      YEAR = {2013},
     PAGES = {114--127},
      ISSN = {0001-8708,1090-2082},
   MRCLASS = {14C20 (13A15 13F20)},
  MRNUMBER = {3091802},
MRREVIEWER = {Ciro\ Ciliberto},
       DOI = {10.1016/j.aim.2013.05.027},
       URL = {https://doi.org/10.1016/j.aim.2013.05.027},
}

@article {BCSGHJNSVV2016,
    AUTHOR = {Bocci, Cristiano and Cooper, Susan and Guardo, Elena and
              Harbourne, Brian and Janssen, Mike and Nagel, Uwe and
              Seceleanu, Alexandra and Van Tuyl, Adam and Vu, Thanh},
     TITLE = {The {W}aldschmidt constant for squarefree monomial ideals},
   JOURNAL = {J. Algebraic Combin.},
  FJOURNAL = {Journal of Algebraic Combinatorics. An International Journal},
    VOLUME = {44},
      YEAR = {2016},
    NUMBER = {4},
     PAGES = {875--904},
      ISSN = {0925-9899,1572-9192},
   MRCLASS = {13F20 (13A02 13F55 14N05 14N20)},
  MRNUMBER = {3566223},
MRREVIEWER = {Christopher\ A.\ Francisco},
       DOI = {10.1007/s10801-016-0693-7},
       URL = {https://doi.org/10.1007/s10801-016-0693-7},
}

@article {BocciCooperGuardoHarbourneJanssenNagelSeceleanuVanTuylVu,
    AUTHOR = {Bocci, Cristiano and Cooper, Susan and Guardo, Elena and
              Harbourne, Brian and Janssen, Mike and Nagel, Uwe and
              Seceleanu, Alexandra and Van Tuyl, Adam and Vu, Thanh},
     TITLE = {The {W}aldschmidt constant for squarefree monomial ideals},
   JOURNAL = {J. Algebraic Combin.},
  FJOURNAL = {Journal of Algebraic Combinatorics. An International Journal},
    VOLUME = {44},
      YEAR = {2016},
    NUMBER = {4},
     PAGES = {875--904},
      ISSN = {0925-9899,1572-9192},
   MRCLASS = {13F20 (13A02 13F55 14N05 14N20)},
  MRNUMBER = {3566223},
MRREVIEWER = {Christopher\ A.\ Francisco},
       DOI = {10.1007/s10801-016-0693-7},
       URL = {https://doi-org.nmsu.idm.oclc.org/10.1007/s10801-016-0693-7},
}

@article {MNDW11,
    AUTHOR = {Mayhew, D. and Newman, M. and Welsh, D. and Whittle, G.},
     TITLE = {On the asymptotic proportion of connected matroids},
   JOURNAL = {European J. Combin.},
  FJOURNAL = {European Journal of Combinatorics},
    VOLUME = {32},
      YEAR = {2011},
    NUMBER = {6},
     PAGES = {882--890},
      ISSN = {0195-6698,1095-9971},
   MRCLASS = {05B35 (05A16 05C80)},
  MRNUMBER = {2821559},
MRREVIEWER = {Vania\ D.\ Mascioni},
       DOI = {10.1016/j.ejc.2011.01.016},
       URL = {https://doi.org/10.1016/j.ejc.2011.01.016},
}

@article {GHMN17,
    AUTHOR = {Geramita, A. V. and Harbourne, B. and Migliore, J. and Nagel, U.},
     TITLE = {Matroid configurations and symbolic powers of their ideals},
   JOURNAL = {Trans. Amer. Math. Soc.},
  FJOURNAL = {Transactions of the American mathematical Society},
    VOLUME = {369},
      YEAR = {2017},
     PAGES = {7049--7066},
}

@misc{DFK26,
    AUTHOR = {{DiPasquale}, Michael and {Fouli}, Louiza and {Kumar}, Arvind},
        title = "{Asymptotic Resurgence of Facet and Stanley-Reisner ideals of Matroids}",
         howpublished = {Available at \url{https://arxiv.org/abs/2607.20892}},
}

@article {HT-2019,
    AUTHOR = {H\`a, Huy T\`ai and Trung, Ngo Viet},
     TITLE = {Membership criteria and containments of powers of monomial
              ideals},
   JOURNAL = {Acta Math. Vietnam.},
  FJOURNAL = {Acta Mathematica Vietnamica},
    VOLUME = {44},
      YEAR = {2019},
    NUMBER = {1},
     PAGES = {117--139},
      ISSN = {0251-4184,2315-4144},
   MRCLASS = {13P10 (05C65 13C05 13F20 90C27)},
  MRNUMBER = {3935294},
MRREVIEWER = {Indranath\ Sengupta},
       DOI = {10.1007/s40306-018-00325-y},
       URL = {https://doi-org.nmsu.idm.oclc.org/10.1007/s40306-018-00325-y},
}

@manual{sagemath,
  Key          = {SageMath},
  Author       = {{The Sage Developers}},
  Title        = {{S}ageMath, the {S}age {M}athematics {S}oftware {S}ystem ({V}ersion 10.8)},
  note         = {{\tt https://www.sagemath.org}},
  Year         = {2025},
}

@incollection {DDGHN2018,
    AUTHOR = {Dao, Hailong and De Stefani, Alessandro and Grifo, Elo\'isa
              and Huneke, Craig and N\'u\~nez-Betancourt, Luis},
     TITLE = {Symbolic powers of ideals},
 BOOKTITLE = {Singularities and foliations. geometry, topology and
              applications},
    SERIES = {Springer Proc. Math. Stat.},
    VOLUME = {222},
     PAGES = {387--432},
 PUBLISHER = {Springer, Cham},
      YEAR = {2018},
      ISBN = {978-3-319-73639-6; 978-3-319-73638-9},
   MRCLASS = {13F20 (14N05)},
  MRNUMBER = {3779569},
MRREVIEWER = {Cleto\ B.\ Miranda-Neto},
       DOI = {10.1007/978-3-319-73639-6\_13},
       URL = {https://doi-org.nmsu.idm.oclc.org/10.1007/978-3-319-73639-6_13},
}

@article {BH2010,
    AUTHOR = {Bocci, Cristiano and Harbourne, Brian},
     TITLE = {The resurgence of ideals of points and the containment
              problem},
   JOURNAL = {Proc. Amer. Math. Soc.},
  FJOURNAL = {Proceedings of the American Mathematical Society},
    VOLUME = {138},
      YEAR = {2010},
    NUMBER = {4},
     PAGES = {1175--1190},
      ISSN = {0002-9939,1088-6826},
   MRCLASS = {14C20 (14N05)},
  MRNUMBER = {2578512},
       DOI = {10.1090/S0002-9939-09-10108-9},
       URL = {https://doi-org.nmsu.idm.oclc.org/10.1090/S0002-9939-09-10108-9},
}

@article {BDHHSS2019,
    AUTHOR = {Bauer, Thomas and Di Rocco, Sandra and Harbourne, Brian and
              Huizenga, Jack and Seceleanu, Alexandra and Szemberg, Tomasz},
     TITLE = {Negative curves on symmetric blowups of the projective plane,
              resurgences, and {W}aldschmidt constants},
   JOURNAL = {Int. Math. Res. Not. IMRN},
  FJOURNAL = {International Mathematics Research Notices. IMRN},
      YEAR = {2019},
    NUMBER = {24},
     PAGES = {7459--7514},
      ISSN = {1073-7928,1687-0247},
   MRCLASS = {14C20 (14N05 14N20)},
  MRNUMBER = {4043827},
MRREVIEWER = {Carlos\ Galindo},
       DOI = {10.1093/imrn/rnx329},
       URL = {https://doi-org.nmsu.idm.oclc.org/10.1093/imrn/rnx329},
}

@article {DHNSST2015,
    AUTHOR = {Dumnicki, M. and Harbourne, B. and Nagel, U. and Seceleanu, A.
              and Szemberg, T. and Tutaj-Gasi\'nska, H.},
     TITLE = {Resurgences for ideals of special point configurations in
              {$\bold{P}^N$} coming from hyperplane arrangements},
   JOURNAL = {J. Algebra},
  FJOURNAL = {Journal of Algebra},
    VOLUME = {443},
      YEAR = {2015},
     PAGES = {383--394},
      ISSN = {0021-8693,1090-266X},
   MRCLASS = {13F20 (13A02 14N05)},
  MRNUMBER = {3400406},
       DOI = {10.1016/j.jalgebra.2015.07.022},
       URL = {https://doi-org.nmsu.idm.oclc.org/10.1016/j.jalgebra.2015.07.022},
}
\bibliographystyle{plain}

\end{document}